\documentclass{amsart}
\usepackage{amsmath,amsthm}
\usepackage{amsfonts,amssymb}
\usepackage{accents}
\usepackage{enumerate}
\usepackage{accents,color}
\usepackage{graphicx}
\usepackage{wrapfig}
\newcommand{\lvt}{\left|\kern-1.35pt\left|\kern-1.3pt\left|}
\newcommand{\rvt}{\right|\kern-1.3pt\right|\kern-1.35pt\right|}

\newtheorem{thm}{Theorem}[section]
\newtheorem{cor}[thm]{Corollary}
\newtheorem{lem}[thm]{Lemma}
\newtheorem{prop}[thm]{Proposition}

\newtheorem{defn}[thm]{Definition}

\theoremstyle{remark}
\newtheorem{rem}{Remark}[section]

 \def\la{{\langle}}
 \def\ra{{\rangle}}

\def\bkappa{{\boldsymbol\kappa}}
\def\brho{{\boldsymbol\rho}}

 \def\sph{{\mathbb{S}^{d-1}}}

 \def\ssb{{\mathsf b}}
 \def\sh{{\mathsf h}}
 \def\su{{\mathsf u}}
 
 \def\sw{{\mathsf w}}

 \def\sJ{{\mathsf J}}
 \def\sK{{\mathsf K}}
 \def\sL{{\mathsf L}}
 
 \def\sP{{\mathsf P}}
 \def\sQ{{\mathsf Q}}
 
 \def\sS{{\mathsf S}}
 \def\sT{{\mathsf T}}

 \def\d{{\mathrm{d}}}
 \def\e{{\mathrm{e}}}
 \def\i{{\mathrm{i}}}
 
 \def\a{{\alpha}}
 \def\b{{\beta}}
 \def\g{{\gamma}}
 \def\k{{\kappa}}
 \def\t{{\theta}}
 \def\l{{\lambda}}
 
 \def\s{\sigma}
 \def\la{{\langle}}
 \def\ra{{\rangle}}

 \def\bb{{\mathbf b}}
 \def\hb{{\mathbf h}}
 \def\kb{{\mathbf k}}

 \def\ub{{\mathbf u}}

 \def\Pb{{\mathbf P}}
 \def\Kb{{\mathbf K}}
 \def\Lb{{\mathbf L}}
 \def\Sb{{\mathbf S}}
 \def\Tb{{\mathbf T}}
 
 \def\CD{{\mathcal D}}
 
 \def\CH{{\mathcal H}}

 \def\CL{{\mathcal L}}

 \def\CV{{\mathcal V}}

 \def\BB{{\mathbb B}}

 \def\NN{{\mathbb N}}

 \def\RR{{\mathbb R}}
 \def\SS{{\mathbb S}}
 
 \def\VV{{\mathbb V}}
 
 \def\ZZ{{\mathbb Z}}
      \def\proj{\operatorname{proj}}

\def\lla{\langle{\kern-2.5pt}\langle}      
\def\rra{\rangle{\kern-2.5pt}\rangle}      

\newcommand{\wt}{\widetilde}
\newcommand{\wh}{\widehat}

\DeclareMathOperator{\esssup}{ess\,sup}

\def\sph{\mathbb{S}^{d-1}}
\def\f{\frac}

\graphicspath{{./}}

\begin{document}
 
\title{Laguerre expansions on conic domains}
\author{Yuan Xu}
\address{Department of Mathematics\\ University of Oregon\\
    Eugene, Oregon 97403-1222.}\email{yuan@uoregon.edu}

\date{\today}

\begin{abstract} 
We study the Fourier orthogonal expansions with respect to the Laguerre type weigh functions on the
conic surface of revolution and the domain bounded by such a surface. The main results include a 
closed form formula for the reproducing kernels 
of the orthogonal projection operator and a pseudo convolution structure on the conic domain; the latter 
is shown to be bounded in an appropriate $L^p$ space and used to study mean convergence of the 
Ces\`aro means of the Laguerre expansions on conic domains.
\end{abstract}
\maketitle

\section{Introduction}
\setcounter{equation}{0}

The Laguerre expansions usually mean the Fourier orthogonal expansions in the Laguerre polynomials 
on the half-line $\RR_+ = [0,\infty)$ or expansions in product Laguerre polynomials on $\RR_+^d$ for 
$d \ge 2$ (cf. \cite{Th}). In the present paper, we study the Fourier orthogonal expansions on the conic surface 
$$
   \VV_0^{d+1} = \left\{(x,t): \|x\| = t, \quad x \in \RR^d, t \in \RR_+\right\}
$$
with respect to the polynomials that are orthogonal with respect to the weight function 
$$
    \sw_{\k}(x,t) = h_\k^2(x)  t^{-1} \e^{-t} \quad\hbox{with} \quad h_\k(x) = \prod_{k=1}^d |x_i|^{\k_i}, 
     \qquad (x,t) \in \VV_0^{d+1},
$$
as well as the expansions on the solid cone $\VV^{d+1}$ bounded by the surface $\VV_0^{d+1}$. 
This is the first study for the Laguerre type weight on conic domains; our results are new even for 
$\sw_0(x,t) = t^{-1} e^{-t}$. 

If $f(x,t) = f_0(t)$ with $f_0: \RR_+\mapsto \RR$, then the orthogonal expansions of $f$ on the conic 
domain becomes the classical Laguerre expansions on $\RR_+$, which has been studied extensively 
in the literature; see  \cite{AW, GM, Ma, Me, MW, Po, Th0, Th} and, for some more recent works, see 
\cite{BRT, NS, NSS, RT} and the references therein. As in the study of classical Laguerre expansions, 
our study relies on intrinsic properties of orthogonal polynomials that hold only for particular weight 
functions and special domains. The study of orthogonal structure on the conic domains is initiated 
recently in \cite{X20a, X20b}. A basis of orthogonal polynomials with respect to $\sw_\k$, called the 
Laguerre type polynomials below, can be given in terms of the classical Laguerre polynomials and 
spherical $h$-harmonics, the latter are orthogonal polynomials with respect to $h_\k^2(x)$ on the unit 
sphere $\sph$ and are the simplest examples in the Dunkl's theory of $h$-harmonics associated with 
reflection groups. This basis is briefly treated in \cite{X20a}, where the main results are established for 
the Jacobi type polynomials on compact conic domains, for which $t\in \RR_+$ is replace by $t \in [0,1]$ 
and $\e^{-t}$ is replace by $(1-t)^\g$. Our starting point is to show that the Laguerre type polynomials with 
respect to $ \sw_{\k}$ on the conic surface can be regarded as limits of the Jacobi type polynomials on 
the compact conic surface, which allows us to derive properties of the former from those of the latter. 
In particular, starting from the closed formula for the reproducing kernels of the Jacobi type polynomials 
in \cite{X20a}, we are able to deduce closed form formulas for the Poisson kernel and for the reproducing 
kernels of the Laguerre type polynomials on the conic domains. Since the reproducing kernels are kernels
of the orthogonal projection operators, the closed form formula provides an essential tool for our analysis. 
It leads to the definition of a pseudo convolution structure on the conic surface, defined via a generalized 
translation operator, which is shown to be bounded in an appropriate $L^p$ space. As an application, 
we use the convolution structure to study the mean convergence of the Ces\`aro $(C,\delta)$ means of 
the Laguerre expansions on the conic surface and determine the critical index for the convergence. We 
will also establish analogues results on the solid cone $\VV^{d+1}$ with respect to the weight function 
$W_{\k,\mu}(x,t) = h_\k^2(x) (t^2-\|x\|^2)^{\mu - \f12} \e^{-t}$, which can be deduced from a closed relation 
between the orthogonal structures on the two domains, akin to the relation between the unit ball and the
unit sphere \cite{DX}. 

The above narrative of our main results on the conic surface is rooted in the study of the Laguerre  
expansions on $\RR_+^d$. The difference in the domain, however, means different geometry and different 
obstacles. The closed form formula for the reproducing kernel, hence the generalized translation operator, 
is substantially more involved on the conic domains. Because of the complexity of the closed form formula, 
proving the boundedness of the pseudo convolution becomes far more elaborate. There is also a hidden
subtle dependence on the parameters, which shows, in the case of $\sw_0$, that the convolution has a 
weaker upper bound for $d = 2$ than that for $d \ge 3$. The analysis on the conic surface is not a 
straightforward extension of the ordinary Laguerre expansions. 

The paper is organized as follows. In the next section we collect results on classical orthogonal polynomials 
and special functions that will be needed. The Laguerre type polynomials and the closed form formulas for 
the Poisson and reproducing kernels will be studied in the the third section. The generalized translation 
operator and the pseudo convolution will be defined in the fourth section, where they are shown to be 
bounded and used to study the Ces\`aro means. Finally, the results on the conic surface are extended to 
solid cone in the fifth section.

\section{Preliminary}
\setcounter{equation}{0}

We recall several families of orthogonal polynomials that will be needed in this section. These are
the Laguerre polynomials, the Gegenbauer and Jacobi polynomials, and spherical $h$-harmonics 
on the unit sphere and orthogonal polynomials on the unit ball. 

\subsection{Laguerre polynomials}
For $\a > -1$, the Laguerre polynomials $L_n^\a$ are given by 
$$
   L_n^\a(t) = \frac{(\a+1)_n}{n!} \sum_{k=0}^n \frac{(-n)_k}{k! (\a+1)_k} t^k, 
$$
and they can also be defined by their generating function 
\begin{equation}\label{eq:generatingL}  
   \sum_{n=0}^\infty L_n^\a (t) r^n = \frac{1}{(1-r)^{\a+1}} \e^{- \frac{t r}{1-r}}.
\end{equation}
The Laguerre polynomials are orthogonal with respect to the inner product 
$$
     \la f, g\ra_\a =b_\a \int_0^\infty f(t) g(t) w_\a(t) \d t, \quad w_\a(t) := t^\a e^{-t}, \quad b_\a :=\frac{1}{\Gamma(\a+1)},
$$
which is normalized so that $\la 1,1\ra_\a =1$. More precisely, for $n, m \in \NN_0$, 
$$
  \la L_n^\a, L_m^\a \ra_\a = h_n^\a \delta_{n,m},  \quad \hbox{where} \quad h_n^\a := \frac{(\a+1)_n}{n!} = L_n^\a(0). 
$$

For $f \in L^2(\RR_+, w_\a)$, its Fourier-Laguerre expansion is defined by 
\begin{equation}\label{eq:FourierLaguerre}
  f(t) = \sum_{n=0}^\infty \wh f_n^\a \frac{L_n^\a(t)}{\sqrt{h_n^\a}}, \quad \hbox{where} \quad \wh f_n^\a = 
     \frac{1}{\sqrt{h_n^\a}} \la f, L_n^\a \ra_\a. 
\end{equation}
One of the important tool for studying the Laguerre expansions is a convolution structure motivated by 
the product formula of the Laguerre polynomials. Let $J_\a$ 
be the usual Bessel function and let  $j_{\a}(z) := z^{-\a} J_\a(z)$. For $\a > -\f12$, the product formula is 
classical \cite{Wa}; see also \cite[p. 142]{Th}, 
\begin{align} \label{eq:productL1}
 \frac{L_n^\a(t) L_n^\a(s)}{L_n^\a(0)} = 
     &\frac{2^{\a} \Gamma(\a+1)}{\sqrt{2\pi}} \int_0^\pi L_n^{\a} \left(t +s + 2 \sqrt{t s} \cos \t \right)  \\
       &\times \e^{- \sqrt{t s}  \cos \t}  j_{\a - \f{1}{2}} \left(\sqrt{t s} \sin \t\right) (\sin \t)^{2 \a}\d\t,  \notag
\end{align}
whereas for $\a = -\f12$, the formula is given in \cite{Ma} and attributed to J. Boersma,
\begin{align} \label{eq:productL2}
& \frac{L_n^{-\f12}(t) L_n^{-\f12} (s)}{L_n^{-\f12}(0)}  = \frac12 \left[\e^{- \sqrt{t s}} L_n^{-\f12} (t + s +2 \sqrt{ts})+
\e^{ \sqrt{ts }} L_n^{-\f12} (t + s - 2 \sqrt{ts}) \right] \\ 
    & \qquad\qquad - ts  \int_0^\pi L_n^{-\f12} \left(t +s + 2 \sqrt{ts} \cos \t \right)  \e^{- \sqrt{ts}  \cos \t} 
      j_1 \left(\sqrt{ts} \sin \t\right)  \sin \t \d\t.  \notag
\end{align}
The convolution is defined as an integral operator with the kernel $T_t g(s)$, where $T_t$ is defined by
$T_t L_n^\a(s) = L_n^\a(t)L_n^\a(s) /L_n^\a(0)$ for all $n \ge 0$. The convolution is bounded in the space 
$L^p_\a(\RR_+)$ that has the norm defined by \cite{GM}
$$
      \|g\|_{p, u(\a)} =  \bigg( b_\a \int_{\RR_+} \left |g(t) \e^{-t/2} \right |^p t^{\a} \d t \bigg)^{\f1 p},  
         \quad 1 \le p < \infty,
$$
where $b_\a = 1/\Gamma(\a+1)$ and, for $p = \infty$, 
$$
\|g\|_{\infty} = \|g\|_{\infty, u(\a)} : = \esssup \left\{|g(t)| \e^{- t/2}: t > 0\right\}.
$$ 
For $p =2$, the norm $ \|\cdot\|_{2, u(\a)}$ coincides with the usual norm of $L^2(\RR_+, w_\a)$. 

For the summability of the Laguerre expansions, one often considers the Ces\`aro mans. 
For $\delta \in \RR$, the Ces\`aro means of the sequence $\{a_n\}_{n=0}^\infty$ are defeind by 
$$
  s_n^\delta = \frac{1}{A_n^\delta} \sum_{k=0}^n A_{n-k}^{\delta} a_k, \qquad A_n^\delta =\binom{n+\delta}{n}, 
  \quad n =1,2,\ldots, 
$$
which can also be defined by the relation 
\begin{equation} \label{eq:Cesaro}
   \frac{1}{(1-r)^{\delta+1}} \sum_{n=0}^\infty a_n r^n = \sum_{n=0}^\infty A_n^\delta s_n^\delta r^n.
\end{equation}
We denote the $(C,\delta)$ means of the Fourier-Laguerre series by $s_n^\delta (w_\a; f)$. The case 
$\delta =0$ is the partial sum operator, which we denote by $s_n (\varpi_\a; f)$; that is, 
$$
   s_n (w_\a; f) =  \sum_{k=0}^n  f_n^\a L_n^\a \qquad \hbox{and}\qquad s_n^\delta(w_\a; f)  =  
             \frac{1}{A_n^\delta} \sum_{k=0}^n A_{n-k}^{\delta} f_k^\a L_k^\a. 
$$ 

For $\a > -1$, the Poisson kernel of the Laguerre polynomials satisfies
\begin{equation}\label{eq:Mehler}
   \sum_{n=0}^\infty \frac{L_n^\a(x) L_n^\a(y)}{\binom{n+\a}{n}} r^n 
   = \frac{\Gamma(\a+1)}{(1-r)^{\a+1}} \e^{- \frac{(x+y) r}{1-r}}(x y r)^{-\f \a 2} I_\a\left(\frac{2 \sqrt{xy r}}{1-r}\right),
\end{equation}
where $I_\a(z) = e^{- \i \a \pi /2}J_\a (z e^{\i \pi /2})$ is the modified Bessel function of the first kind 
\cite[(10.27.6)]{DLMF}. For $\a \ge - \f12$, $I_\a$ has an integral representation \cite[(10.32.4)]{DLMF},
which leads to the following formula for the Poisson kernel: for $\a \ge - \f12$, 
\begin{equation}\label{eq:Mehler2}
   \sum_{n=0}^\infty \frac{L_n^\a(x) L_n^\a(y)}{\binom{n+\a}{n}} r^n 
   = \frac1{(1-r)^{\a+1}} \e^{- \frac{(x+y) r}{1-r}} c_\a \int_{-1}^1 \e^{\frac{2 \sqrt{x y r}}{1-r} v} (1-v^2)^{\a-\f12} \d v,
\end{equation}
where the constant $c_\a$ is given by 
\begin{equation}\label{eq:cl}
c_\a = \frac{\Gamma(\a+1)}{\Gamma(\f12) \Gamma(\a+\f12)}, \qquad \a > -\tfrac 12, 
\end{equation}
and the formula holds under the limit for $\a = -\f12$ using 
\begin{equation}\label{eq:limitInt}
 \lim_{a \to \f12 +} c_\a \int_{-1}^1 f(t) (1-t^2)^{a-\f12} \d t  = \f {f(1) + f(-1)}2. 
\end{equation}

Finally, we will need the following lemma \cite[Lemma 1]{Ma} or \cite[Lemma 1.5.4]{Th}.

\begin{lem}\label{lem:estLn}
Let $\a + \b > -1$. Then
$$
 b_{\a} \int_0^\infty \left |L_n^{\a+\b}(t) \right | t^{\a/2} \e^{-t/2} \d t 
   \sim \begin{cases} n^{\f{\a+1}2}, & \hbox{if $\beta <  \f32$} \\
   n^{\f{\a+1}2}\log n, &  \hbox{if $\beta = \f32$} \\
    n^{\f{\a}2 + \b -1},  &  \hbox{if $\beta > \f32$}.          \end{cases}
$$
\end{lem}

\subsection{Jacobi and Gegenbauer Jacobi polynomials}
We will also need the Jacobi polynomials, denoted by $P_n^{(\a,\b)}$, that satisfy the orthogonal relation
$$
   c'_{\a,\b} \int_{-1}^1 P_n^{(\a,\b)} (t) P_m^{(\a,\b)} (t) (1-t)^\a (1+t)^\b \d t = h_n^{(\a,\b)} \delta_{n,m}, \quad
     \a,\b> -1,
$$
where the normalization constant $c'_{\a,\b} = 2^{-\a-\b-1}c_{\a,\b}$ with 
\begin{equation}\label{eq:c_ab}
   c_{\a,\b} = \frac{\Gamma(\a+\b+2)}{\Gamma(\a+1)\Gamma(\b+1)}
\end{equation}
and $h_n^{(\a,\b)}$, the square of the norm of $P_n^{(\a,\b)}$, is given by 
$$
  h_n^{(\a,\b)} = \frac{(\a+1)_n(\b+1)_n (\a+\b+n+1)}{n! (\a+\b+2)_n (\a+\b+2n+1)}.
$$

The Gegenbauer polynomials, denoted by $C_n^\l$, satisfy the orthogonal relation
$$
   c_\l  \int_{-1}^1 C_n^\l (t) C_m^\l (t) (1-t^2)^{\l-\f12} \d t =  \frac{\l}{n+\l}C_n^(1) \delta_{n,m},   \quad \l > \tfrac12,
$$
and normalized by $C_n^\l(1) = (2\l)_n/n!$, where the constant $c_\l$ is as in \eqref{eq:cl}. When $\l = 0$, the Gegenbauer polynomial reduces to the Chebyshev polynomial $T_n$ of the first kind, 
defined by $T_n(\cos \t) = \cos n \t$. For convenience, we define 
\begin{equation}\label{eq:Zn}
   Z_n^\l (t):= \frac{n+\l}{\l} C_n^\l(t), \quad \l > 0, \quad \hbox{and} \quad 
         Z_n^0(t):= \begin{cases}2 T_n(t) & n \ge 1 \\ 1 & n =0 \end{cases},
\end{equation}
which will be used multiple times in our discussion below. 

\subsection{Orthogonal polynomials on the unit sphere and the unit ball}
For $\k = (\k_1,\ldots,\k_d) \in \RR^d$, we define the weight function
\begin{equation}\label{eq:hk}
  h_\k(x) = \prod_{k=0}^d |x_i |^{\k_i}, \qquad \k_i \ge 0, \quad x\in \RR^d,
\end{equation}
which is invariant under the group $\ZZ_2^d$ of sign changes. On the unit sphere $\sph$ of $\RR^d$,
we consider orthogonal polynomials with respect to the inner product
$$
  \la f, g\ra_{\SS} = c_\k^h \int_{\sph} f(\xi) g(\xi) h_\k^2(\xi)\d \s_\SS (\xi), 
  \qquad c_\k^h:= \frac{\Gamma(|\k|+\f d2)}{2 \Gamma(\k_1+\f 12)\cdots  \Gamma(\k_d+\f 12)},
$$
where $\d \s_\SS$ denotes the surface measure of $\sph$, $|\k| = \k_1+\cdots + \k_d$, and it is 
normalized so that $\la 1,1\ra_\SS = 1$. The orthogonal polynomials for this inner product are a special 
case of Dunkl's $h$-harmonics for weight functions invariant under reflection groups. Let $\CH_n(h_\k^2)$ 
be the of $h$-harmonics of degree $n$ in $d$ variables. It consists of homogeneous polynomials $Y$ 
of degree $n$ in $d$-variables that satisfy $\Delta_h Y =0$, where $\Delta_h$ is the Dunkl Laplacian, a 
second order differential-difference operator \cite{D89}, which we shall not specify. If $Y \in \CH_n^d(h_\k^2)$, 
then $Y(x) = r^n Y(\xi)$, $x = r \xi$ with $\xi \in \sph$. We shall also use $\CH_n^d(h_\k^2)$ to denote the 
space of $Y(\xi)$, the restriction of $h$-harmonics on the unit sphere or spherical $h$-harmonics.
When $\k =0$, or $h_\k(x) =1$, $\Delta_h$ becomes the ordinary Laplace operator $\Delta$ and the 
space $\CH_n^d(h_0^2)$ becomes the space $\CH_n^d$ of ordinary spherical harmonics of degree $n$
in $d$ variables. It is known that 
$$
    a_n^d : = \dim \CH_n^d = \binom{n+d-1}{n} - \binom{n+d-3}{n-2}.
$$ 

The spherical $h$-harmonics are orthogonal polynomials in the sense that $\la Y_n, Y_m \ra_\SS =0$ 
if $Y_n$ and $Y_m$ are spherical $h$-harmonics of degree $n$ and $m$ with $n\ne m$. An explicit 
orthogonal basis of $\CH_n^d(h_k^2)$ can be given in terms of the Jacobi polynomials (cf. \cite[p. 229]{DX}).
The reproducing kernel of $\CH_n^d(h_\k^2)$ satisfies a closed form formula, which is a generalization 
of the addition formula for the ordinary spherical harmonics.  For $n = 0,1,2,\ldots$, let 
$\sP_n(h_\k^2; \cdot,\cdot)$ be the reproducing kernel of $\CH_n^d(h_\k^2)$. Then \cite{X97b}
\begin{equation} \label{eq:reprod-h}
  \sP_n(h_\k^2; \xi,\eta) = V_\k \left[ Z_n^{|\k|+ \f{d-2}{2}}\la \xi, \cdot\ra \right] (\eta), \quad \xi,\eta\in \sph,
\end{equation}  
where $V_\k$ is the integral operator, called intertwining operator, defined by 
\begin{equation} \label{eq:intertwin}
   V_\k f(x) = c_{\k- \frac{\mathbf{1}}2} \int_{[1,1]^d}  
                                       f (x_1 u_1, \cdots, x_d   u_d) \Phi_\k(u) \d u,
\end{equation}
where $c_{\k- \frac{\mathbf{1}}2} = c_{\k_1-\f12} \cdots c_{k_d -\f12}$ with $c_\l$ as defined 
in \eqref{eq:cl} and 
\begin{equation} \label{eq:Phi} 
  \Phi_\k(u) = \prod_{i=1}^d (1+u_i) (1-u_i^2)^{\k_i -1}. 
\end{equation}
If $\k_i = 0$ for some $\k_i$, then the identity \eqref{eq:intertwin} is defined under the limit \eqref{eq:limitInt}. 
In particular, if $\k_i = 0$ for all $i$, then $h_\k(x) = 1$ and the right-hand side of \eqref{eq:reprod-h} 
reduces to $Z_n^{\f{d-2}{2}}(\la x, y\ra)$, in which case \eqref{eq:reprod-h} degenerates to the addition 
formula of the ordinary spherical harmonics. We need the following lemma proved in \cite{X97b}.

\begin{lem}
Let $g: \RR\mapsto \RR$ be a function such that both integrals below are defined. Let 
$\a_\k = |k| + \f{d-2}{2}$. Then, for $x \in \RR^d$,
\begin{equation} \label{eq:intVk} 
  c_\k^h \int_{\sph} V_\k [f(\la x,\cdot \ra)] (\eta) h_\k^2(\eta) d\s_\SS (\eta) = 
       c_{\a_\k}\int_{-1}^1  f(\|x\| v)(1-v^2)^{\a_\k -\f12}\d v. 
\end{equation}
\end{lem}

The closed formula \eqref{eq:reprod-h} has been instrumental for recent progress on analysis on the 
sphere equipped with $h_\k^2$ weight; see, for example, \cite{DaiX,DX}.

\subsection{Orthogonal polynomials on the unit ball}
It is known \cite[Section 4.4]{DX} that orthogonal polynomials on the unit ball is closely related to orthogonal 
polynomials on the unit sphere. Let $h_\k$ be defined in \eqref{eq:hk}. For the unit ball $\BB^d$ of $\RR^d$, 
we let 
$$
       \varpi_{\k,\mu} (x) = h_\k^2(x) (1-\|x\|^2)^{\mu-\f12}, \quad  \mu> -\tfrac12, \quad  x \in \BB^d,
$$
and consider orthogonal polynomials with respect to the inner product 
$$
  \la f,g\ra_{\BB} = b_{\k,\mu} \int_{\BB^d} f(x) g(x) \varpi_{\k,\mu}(x) \d x,
$$
where $b_{\k,\mu}$ is a normalization constant such that $\la 1,1\ra_{\BB} =1$. For $\k \in \RR^d$, 
we define $\k_{d+1} = \mu$ and $\bkappa = (\k,\k_{d+1}) \in \RR^{d+1}$. For $x \in \BB^d$, we define 
$x_{d+1} = \sqrt{1-\|x\|^2}$ and $X = (x, x_{d+1})$ as well as $X^- = (x, x_{d+1})$; then 
$X \in \SS^{d}$ and $X^- \in \SS^{d}$. It is easy to verify that the measure $\varpi_{\k,\mu} \d x $ on $\BB^d$ 
satisfies 
$$
   \varpi_{\k,\mu}(x) \d x = \prod_{i=1}^{d+1} |x_i |^{2\k_i} d\s(X) = h_\bkappa^2(X)\d \s_{\SS^d} (X), 
$$
where $\d\s_{\SS^d}$ is the surface measure of $\SS^d$. By symmetry, it follows that 
$$
   \int_{\BB^d} f(x) \varpi_{\k,\mu}(x) \d x  =  \f12 \int_{\SS^d} \left[f(X) + f(X^-) \right] h_\bkappa^2(X)\d \s_{\SS^d}(X).
$$
Let $\CV_n(\BB^d, \varpi_{\k,\mu})$ be the space of orthogonal polynomials with respect to the inner 
product $\la \cdot,\cdot\ra_\BB$. We then have the decomposition
$$
    \CH_n^{d+1}(h_\bkappa) = \CV_n(\BB^d, \varpi_{\k,\mu}) \bigoplus x_{d+1} \CV_{n-1}(\BB^d, \varpi_{\k,\mu+1}).
$$
In particular, a basis for $\CV_n(\BB^d, \varpi_{\k,\mu})$ can be derived from those spherical $h$-harmonics 
in $\CH_n^{d+1}(h_\bkappa)$ that are even in the $x_{d+1}$ variable. Moreover, we can also derive a close
formula for the reproducing kernel of $\CV_n(\BB^d, \varpi_{\k,\mu})$ through this correspondence. 

\section{Laguerre expansions on the conic surface}
\setcounter{equation}{0}

We work on the conic surface that is standardized by 
$$
  \VV_0^{d+1} =\left \{(x,t) \in \RR^{d+1}: \|x\| = t, \, x \in \RR^d, t \in \RR_+ \right\}, \quad d \ge 2. 
$$
For $\k \ge 0$, we define the Laguerre type weight function $\sw_{\k}$ on the conic surface by
$$
 \sw_{\k}(x,t) =  h_\k^2(x) t^{-1} \e^{-t}, \qquad  (x,t) \in \VV_0^{d+1}, 
$$
where $h_\k$ is given in \eqref{eq:hk}. In the first subsection, we discuss orthogonal polynomials with 
respect to $\sw_\k$. In the second subsection we derive closed form formulas for the Poisson kernel and 
the reproducing kernels of these polynomials.

\subsection{Orthogonal polynomials}
For $d \ge 2$, we define the inner product with respect to $\sw_{\k}$ on the polynomial space 
$\RR[x,t] / \la \|x\|^2 - t^2\ra$ by 
\begin{align*}
    \la f,g\ra_{\k} &=  \ssb_{\k} \int_{\VV_0^{d+1}} f(x,t) g(x,t)  \sw_{\k}(x,t)  \d \sigma(x,t),
\end{align*}
where $\d \s(x,t)$ denotes the Lebesgue measure on the conic surface and $\ssb_{\k}$ is the normalization 
constant chosen so that $\la 1, 1 \ra_{\k} =1$. The value of $\ssb_\k$ is given by
$$
  \ssb_{\k} = \frac{1}{\int_{\VV_0^{d+1}} \sw_{\k}(x,t)\d \sigma(x,t)} =  \frac{c_\k^h}{\Gamma(2|\k|+ d-1)}
    = b_{2|\k|+d-2} c_\k^h,
$$
evaluated by writing the integral over $\VV_0^{d+1}$ as, using the spherical-polar coordinates,  
\begin{equation} \label{eq:intV0}
    \int_{\VV_0^{d+1}} f(x,t) d\s(x,t) = \int_0^\infty t^{d-1} \int_{\sph} f(t\xi, t) \d\s_\SS (\xi) \d t.
\end{equation}
Let $\CV_n(\VV_0^{d+1},\sw_{\k})$ denote the space of orthogonal polynomials of degree $n$ with 
respect to this inner product. The space has the same dimension as that of $\CH_n^{d+1}$, so that 
$$
   \dim \CV_n(\VV_0^{d+1},\sw_{\k})  = \binom{n+d}{n}- \binom{n+d-2}{n-2}. 
$$
Using \eqref{eq:intV0}, we see that a family of orthogonal polynomials on the conic surface can be
given in terms of $h$-spherical harmonics and the Laguerre polynomials. 

\begin{prop}\label{prop:OPV0}
Let $\{Y_\ell^m: 1\le \ell \le a_m^d\}$ be an orthonormal basis of $\CH_m^d(h_\k^2)$. Define 
\begin{equation}\label{eq:coneLsf}
   \sL_{m,\ell}^{n}(x,t) = L_{n-m}^{2m + 2 |\k| + d-2}(t) Y_\ell^m(x), \quad 0 \le m \le n, \quad 1\le \ell \le a_m^d.
\end{equation}
Then $\{\sL_{m,\ell}^n: 0 \le m \le n, \quad 1\le \ell \le a_m^d\}$ is an orthogonal basis of 
$\CV_n(\VV_0^{d+1},\sw_{\k})$. Moreover, the norm square of $\sL_{m,\ell}^n$ is given by
$$
  \sh_{m,n}: = \la \sL_{m,\ell}^n, \sL_{m,\ell}^n\ra_{\k} = \frac{(2|\k|+d-1)_{n+m}}{(n-m)!}. 
$$
\end{prop}

\begin{proof}
Let $\a = |\k|+ \f{d-2}{2}$. Then, using \eqref{eq:intV0}, we obatin
\begin{align*}
  \la \sL_{m,\ell}^n, \sL_{m',\ell'}^{n'} \ra_{\k,-1} \, & =  b_{2\a} 
       \int_0^\infty L_{n-m}^{2m+ \a}(t) L_{n'-m'}^{2m + \a}(t) t^{2m + 2\a} e^{-t} \d t \\
      &\qquad \quad  \times c_\k^h \int_{\sph} Y_\ell^m(\xi) Y_{\ell'}^{m'} (\xi) h_\k^2(\xi) \d \s_\SS(\xi) \\ 
      & = \frac{\Gamma(2m+2\a+1)}{\Gamma(2\a+1)} h_{n-m}^{2m+2\a} \delta_{n,n'} \delta_{m,m'}  \delta_{\ell,\ell''},
\end{align*}
where $h_{n-m}^\a$ is the norm square of the Laguerre polynomial $L_{n-m}^\a$. 
\end{proof}

We could consider orthogonal polynomials with respect to the weight function $h_\k^2(x) t^\b \e^{-t}$ for
$\b \ne -1$. However, the case $\b = -1$ appears to be the most natural setting for several reasons \cite{X20a}.
In particular, the closed form formula of the reproducing kernels is of a relative simple form for $\b = -1$; 
see Theorem \ref{thm:reprod-closed} below. 

We call the polynomials in \eqref{eq:coneLsf} Laguerre type polynomials when $\k \ne 0$, and Laguerre 
polynomials when $\k = 0$, on the conic surface.  
For $f \in L^2(\VV_0^{d+1}, \sw_{\k})$, the Fourier-Laguerre expansion of $f$ is defined by
$$
   f = \sum_{n=0}^\infty \sum_{m=0}^n \sum_{\ell =1}^{a_m^d} \wh f_{m,\ell}^n \sL_{m,\ell}^n 
   \qquad \hbox{with}\qquad  \wh f_{m,\ell}^n =  \frac{\la f, \sL_{m,\ell}^n\ra_{\k}} {\sh_{m,n}}
$$
The projection operator $\proj_n(\sw_\k): L^2(\VV_0^{d+1}, \sw_{\k}) \mapsto \CV_n(\VV_0^{d+1},\sw_{\k})$ 
and the $n$-th partial sum operator $\sS_n(\sw_\k)$ of this expansion are defined by 
$$
    \proj_n \left(\sw_{\k}; f\right) = \sum_{m=0}^n \sum_{\ell =1}^{a_m^d} \wh f_{m,\ell}^n \sL_{m,\ell}^n
    \quad \hbox{and} \quad  \sS_n f\left(\sw_{\k}; f\right)  = \sum_{m=0}^n \proj_m f\left(\sw_{\k}; f\right).
$$
If $f(x,t)$ depends only on $t$, then $\wh f_{m,\ell}^n = 0$ for all $m > 0$ and $\wh f_{0,\ell}^n$ reduces to the 
classical Fourier-Laguerre coefficients of one variable. In particular, the following 
proposition follows readily. 

\begin{prop}
Let $f(x,t) = f_0(t)$ and $f_0 \in L^2(\RR_+; w_{2|\k|+d-2})$. Then the Fourier-Laguerre expansion of $f$ 
on the conic surface becomes the Laguerre expansion of $f_0$ in $ L^2(\RR_+; w_{2|\k|+d-2})$. In particular, 
$$
     \sS_n \left(\sw_{\k}; f, (x,t)\right) = s_n (w_{2|\k|+d-2}; f_0, t), \qquad n= 0,1,2,\ldots. 
$$
\end{prop}

The projection operator $\proj_n(\sw_{\k})$ of the Fourier-Laguerre series can be written as an 
integral operator
$$
\proj_n(\sw_{\k}) f = \ssb_{\k} \int_{\VV_0^{d+1}} f(y,s) \sP_n\left(\sw_{\k}; (x,t), (y,s)\right) \sw_{\k}(y,s) \d \s (y,s), 
$$
where $\sP_n(\sw_{\k}; \cdot,\cdot)$ is the reproducing kernel of the space $\CV_n (\VV_0^{d+1}, \sw_{\k})$, 
which can be written in terms of the orthogonal basis $\sL_{m,\ell}^n$ as 
$$
 \sP_n\left(\sw_{\k}; (x,t), (y,s)\right) = \sum_{m=0}^n \sum_{\ell=1}^{a_m^d} 
     \frac{\sL_{m,\ell}^n(x,t)\sL_{m,\ell}^n(y,s)}{\sh_{m,n}}.
$$

The Laguerre type polynomials on the conic surface can be derived from the Jacobi type polynomials on 
the conic surface via a limit process. The latter polynomials are orthogonal with respect to the inner product 
$\la \cdot,\cdot\ra_{\sw_{\k,\g}}$ defined by 
$$
  \la f,g\ra_{\k,\g} = \ssb_{\k,\g} \int_{\VV_0^{d+1}} f(x,t) g(x,t)  \sw_{\k,\g}(x,t) h_\k^2(x,t) \d \s(x,t),
$$
where $\ssb_{\k,\g} = c_\k^h c_{2|\k|+d-1,\g}$, with $c_{\a,\b}$ defined in \eqref{eq:c_ab}, and the weight 
function $\sw_{\k,\g}$ is given by 
$$
    \sw_{\k,\g}(t) =  h_\k^2(x) t^{-1} (1-t)_+^\g, \qquad \k_i \ge 0, \quad \g > - 1,
$$ 
where $(1-t)_+^\g = (1-t)^\g$ if $t < 1$ and $0$ if $t \ge 0$. Let $\CV_n(\VV_0^{d+1},\sw_{\k,\g})$
be the corresponding space of orthogonal polynomials of degree $n$. An orthogonal basis of this space 
can be given by the Jacobi polynomials and spherical $h$-harmonics \cite{X20a}.

\begin{prop}\label{prop:OPV0Jacobi}
Let $\{Y_\ell^m: 1\le \ell \le a_m^d\}$ be an orthonormal basis of $\CH_m^d(h_\k^2)$. Define 
\begin{equation}\label{eq:coneJsf}
   \sJ_{m,\ell}^{n}(x,t) = P_{n-m}^{(2m+2|\k|+d-2, \g)}(1-2t) Y_{\ell}^m(x), \quad 1 \le \ell \le a_m^d, \, \, 0 \le m \le n.\end{equation}
Then $\{\sJ_{m,\ell}^n: 0 \le m \le n, \quad 1\le \ell \le a_m^d\}$ is an orthogonal basis of 
$\CV_n(\VV_0^{d+1},\sw_{\k,\g})$. Moreover, the norm square of $\sJ_{m,\ell}^n$ is given by
$$
  \sh^\sJ_{m,n} = \la  \sJ_{m,\ell}^{n}, \sJ_{m,\ell}^{n}\ra_{\sw_{\k,\g} } =\frac{(2|\k|+d-1)_{2m}}{(2|\k|+\g+d)_{2m}} 
      h_{n-m}^{(2m+2|\k|+d-2,\g)}, 
$$
where $h_n^{(\a,b)}$ is the norm square of the Jacobi polynomial of degree $n$. 
\end{prop}

The proof follows exactly as that of Proposition \ref{prop:OPV0} using the Jacobi polynomials. 
We denote the reproducing kernel of $\CV_n\left(\VV_0^{d+1}, \sw_{\k,\g}\right)$ by
$\sP_n^\sJ(\sw_{\k,\g}; \cdot,\cdot)$. Then
$$
  \sP_n^\sJ\big(\sw_{\k,\g}; (x,t), (y,s)\big) = \sum_{m=0}^n \sum_{\ell =1}^{a_m^d}
     \frac{\sJ _{m,\ell}^{n}(x,t) \sJ_{m,\ell}^{n}(y,s)}{\sh^\sJ_{m,n}}.
$$

\begin{prop}
For $d \ge 2$, 
\begin{equation} \label{eq:JtoL}
   \lim_{\g \to \infty} \g^m  \sJ_{m,\ell}^{n}\left( \f{x}{\g},\f{t}{\g} \right) = \sL_{m,\ell}^{n}(x,t) \quad \hbox{and}\quad
    \lim_{\g \to \infty} \g^{2m} \sh^\sJ_{m,n} = \sh_{m,n}. 
\end{equation}
Furthermore, for the reproducing kernels, 
\begin{equation} \label{eq:KernelJtoL}
  \lim_{\g\to \infty} \sP_n^\sJ\left(\sw_{\k,\g}; \left( \f{x}{\g},\f{t}{\g} \right),\left( \f{y}{\g},\f{s}{\g} \right)\right) = 
      \sP_n(\sw_{\k}; (x,t),(y,s)). 
\end{equation}
\end{prop}

\begin{proof} 
Since the Laguerre polynomials are limits of the Jacobi polynomials \cite[(5.3.4)]{Sz},
$$
  \lim_{\b \to \infty}  P_n^{(\a,\b)}(1 - 2 \b^{-1} t)= L_n^\a(t),
$$
the limits in \eqref{eq:JtoL} follow accordingly using $Y_\ell^m( \g^{-1} x) = \g^{-m} Y_\ell^m (x)$ and 
the explicit formula of $h_n^{(\a,\b)}$. The limit \eqref{eq:KernelJtoL} follows as a consequence using 
the expression of both kernels written in terms of their corresponding orthogonal bases. 
\end{proof}

The reproducing kernel of the Jacobi polynomials on the conic surface satisfies a closed form formula,
discovered recently in \cite{X20a}, given via $Z_n^\l$ in \eqref{eq:Zn}.

\begin{thm}  \label{thm:sfPbCone}
Let $d \ge 2$, $\k_i \ge 0$ and $\g \ge -\f12$. Let $\a_\k = |\k| + \frac{d-2}{2}$. Then, for $(x,t), (y,s) \in \VV_0^{d+1}$,
\begin{align} \label{eq:sfPbCone}
 \sP_n^\sJ \big(\sw_{\k,\g}; (x,t), (y,s)\big) = \, & 
   c_{\k,\g} \int_{[-1,1]^{d+2}} Z_{2n}^{ 2 \a_\k +\g+1} \big( \zeta (x,t,y,s; u,v)\big) \\
    &\times (1-v_1^2)^{\a_\k -1}(1-v_2^2)^{\g-\f12} \d v  \Phi_\k(u) \d u, \notag
\end{align} 
where $c_{\k,\g} = c_\g c_{\a-\f12} c_{\k-\f{\mathbf{1}}{2}}$ and $\zeta (x,t,y,s; v) 
= v_1 \rho(x,t,y,s;u) + v_2 \sqrt{1-t}\sqrt{1-s}$ with 
\begin{equation} \label{eq:rho}
        \rho(x,t,y,s; u) :=  \sqrt{\tfrac{1}{2} (t s + x_1y_1 u_1+ \cdots + x_d y _d u_d)};
\end{equation}
moreover, the identity holds under the limit \eqref{eq:limitInt} if either $\g = -\f12$, or $d = 2$, or $\k_i =0$ for
one or more $i$. 
\end{thm} 

The formula \eqref{eq:sfPbCone} is stated in \cite[Theorem 10.2]{X20a} where, however, $|\k|$ should be 
replaced by $2|\k|$. There is also a more general formula when the weight function is $h_\k^2(x)t^\b(1-t)^\g$, 
which however is more involved for our purpose. 

\subsection{Poisson kernel and reproducing kernels}
Our main results in this subsection are the closed form formulas for the Poisson kernels and the reproducing
kernels of orthogonal polynomials on the conic surface. We start with the Poisson kernels 
$\sP(\sw_{\k}; r, \cdot, \cdot)$ defined by 
\begin{equation} \label{eq:PoissonLV0}
    \sQ_r(\sw_{\k}; (x,t), (y,s)) =  \sum_{n=0}^\infty \sP_n (\sw_{\k}; x,t,y,s) r^n, \qquad 0 \le r <1. 
\end{equation}

\begin{thm} \label{thm:PoissonV0L}
Let $d \ge 2$, $\k_i \ge 0$. If $\a_\k = |\k| + \frac{d-2}{2} >0$, then 
\begin{align*}
 \sQ_r(\sw_{\k}; (x,t), (y,s))  = \, & \frac{\e^{ -\frac{(t+s)r}{1-r}}}{(1-r)^{2\a_\k+1}}  
      c_{\a_\k-\f12} c_{\k-\f{\mathbf{1}}{2}} \int_{[-1,1]^{d+1}} \exp\bigg\{\frac{ 2 \sqrt{r}  v }{1-r} \rho(x,t,y,s;u)\bigg\} \\
        & \times  \Phi_\k(u)  \d u (1-v^2)^{\a_\k-1} \d v,
\end{align*}
where $\Phi_\k$ is in \eqref{eq:Phi}; moreover, if $\a_\k = 0$ or, equivalently, $\k =0$ and $d =2$, then
$$
 \sQ_r (\sw_{0}; (x,t), (y,s)) = 
  \frac{\e^{ -\frac{(t+s)r}{1-r}}}{1-r} \cosh \bigg(\frac{2\sqrt{r}}{1-r}\rho(x,t,y,s)\bigg), 
$$
where $\rho(x,t,y,s) = \rho(x,t,y,s;1) =  \sqrt{\tfrac{1}{2} (t s + \la x, y\ra)}$. 
\end{thm}
 
\begin{proof}
We start with \eqref{eq:sfPbCone} and change variable $v_2 \mapsto r =  \zeta (x,t,y,s; u, v_1,v_2)$ in the
integral against $v_2$, so that 
$$
   v_2 = \frac{r - v_1 \rho(x,t,y,s; u)}{\sqrt{1-t}\sqrt{1-s}} \quad \hbox{and} \quad \d v_2 = \frac{1}{\sqrt{1-t}\sqrt{1-s}} \d r.
$$
This leads to a reformulation of $\sP_n^\sJ\big(\sw_{\k,\g}; \cot,\cdot\big)$ as 
\begin{align} \label{eq:PnJ-lem}
\sP_n^\sJ \big(\sw_{\k,\g}; (x,t), (y,s)\big) = \,& c_\g \int_{-1}^1 Z_{2n}^{\l}(r) G_\g(x,t,y,s; u; r) (1-r^2)^{\l -\f12} \d r,
\end{align}
where $\l = \g+ 2\a_\k + 1 = \g+2|\k|+d-1$ and 
$$
     G_\g(x,t,y,s; u; r) = c_{\a_\k-\f12} c_{\k-\f{\mathbf{1}}{2}} 
           \int_{[-1,1]^{d+1}} F_\g(x,t,y,s,r;u,v) (1-v^2)^{\a_\k-1} \d v  \Phi_\k(u)\d u
$$
in terms of the the function $F_\g$ defined by
$$
  F_\g(x,t,y,s,r;u,v) =\frac{1}{\sqrt{1-t}\sqrt{1-s} (1-r^2)^{\l-\f12}} \left(1- \frac{(r - v \rho(x,t,y,s;u))^2}{(1-t)(1-s)} \right)^{\g-\f12} 
$$
if $|r- v  \rho(x,t,y,s;u)| \le \sqrt{1-t}\sqrt{1-s}$ and $F_\g(x,t,y,s,r;u, v) = 0$ otherwise. Let us define 
$$
  \wt  G_\g(x,t,y,s;u,r)  = \frac12 \left[ G_\g(x,t,y,s;u, r) + G_\g(x,t,y,s; u, - r) \right],
$$
which is even in $r$. By symmetry of the integral, it follows from \eqref{eq:PnJ-lem} that 
$$
 c_\g \int_{-1}^1 Z_{m}^{\l}(r) \wt G_\g(x,t,y,s;u, r) (1-r^2)^{\l -\f12} \d r = \begin{cases} 
 \sP_n^\sJ \big(w_{-1,\g}; (x,t), (y,s)\big), &  m =2n, \\
 0, & m = 2n+1.
\end{cases}
$$
The integral in the left-hand side gives the Fourier coefficient, up to the normalization constant, of the function 
$r\mapsto\wt G_\g(x,t,y,s; r)$ of the Fourier-Gegenbauer series in $L^2([-1,1], (1-\{\cdot\}^2)^{\l-\f12})$. 
Consequently, using 
$h_n^\l = \frac{n+\l}{\l}C_n^\l(1)$, it follows that 
$$
 \wt G_\g(x,t,y,s; r) = \frac{c_\g}{c_\l} \sum_{n=0}^\infty  \sP_n^\sJ \big(\sw_{\k,\g}; (x,t), (y,s)\big) \frac{C_{2n}^\l(r)}{C_{2n}^\l(1)}.
$$
Since $\l = \g+2 \a_\k +1$, $c_\g/c_\l \to 1$ as $\g \to \infty$ and, by the ${}_2F_1$ expression of $C_n^\l$, it is easy 
to see that $C_{2n}^\l(r)/ C_{2n}^\l(1) \mapsto r^{2n}$ as $\g \to \infty$, it follows from \eqref{eq:JtoL} that 
\begin{equation} \label{eq:limitGg}
   \lim_{\g \to \infty}  \wt G_\g(x,t,y,s; r) =    \sQ_{r^2} \big(\sw_{\k}; (x,t), (y,s)\big).
\end{equation}
Since $\rho(x/\g,t\g,y\g,s\g; u) = \frac1{\g} \rho(x,t,y,s;u)$, we see that $F_\g$ satisfies 
\begin{align*}
 F_\g \left( \f{x}{\g},  \f{t}{\g},  \f{y}{\g},  \f{s}{\g}, r; u,v \right) =&  \frac{1}{(1-r^2)^{2\a_\k+1}(1-\frac{t}{\g})^\g
      (1-\frac{s}{\g})^\g} \\
     \times & \left(1- \frac{t+s-2 v \rho(x,t,y,s;u)}{\g(1-r^2)} - \frac{t s+ v^2 \rho(x,t,y,s;u)^2}{\g^2(1-r^2)} 
    \right)^{\g-\f12}
\end{align*}
if $r-  \frac{v}{\g} \rho(x,t,y,s;u) \le \sqrt{1-\f t \g}\sqrt{1-\f s \g}$. Taking the limit $\g \to \infty$, we obtain
\begin{align*}
  \lim_{\g \to \infty} F_\g \left( \f{x}{\g},  \f{t}{\g},  \f{y}{\g},  \f{s}{\g}, r; u, v \right)
     = & \frac{1}{(1-r^2)^{2\a_\k+1}\e^{-t -s}} \exp\left\{ -\frac{t+s - 2 v r \rho(x,t,y,s;u)}{1-r^2}\right\} \\
     = & \frac{1}{(1-r^2)^{2\a_\k+1}}  \exp\left\{ -\frac{ (t+s)r^2 - 2 v r \rho(x,t,y,s;u)}{1-r^2}\right\}. 
\end{align*}
Hence, taking the limit of $G_\g$ accordingly, we conclude that 
\begin{align*}
  \lim_{\g \to \infty} G_\g\left( \f{x}{\g},  \f{t}{\g},  \f{y}{\g},  \f{s}{\g}, r; v \right)
      =  \, & \frac{1}{(1-r^2)^{2\a_\k+1}} \e^{ -\frac{(t+s)r^2}{1-r^2}} \int_{[-1,1]^{d+1}}
         \exp\left\{\frac{ 2  r \rho(x,t,y,s;u)}{1-r^2} v \right\} \\
          & \times c_{\a-\f12}c_{\a_\k-\f{\mathbf{1}}2} (1-v^2)^{\a-1} \d v   \Phi_\k(u) \d u.
\end{align*}
Changing variable $v\to -v$ shows that the right-hand side is an even function in $r$. Hence, the 
limit of $\wt G_\g$ is the sam as the limit of $G_\g$. Thus, comparing the above limit with 
\eqref{eq:limitGg}, we complete the proof. 
\end{proof}

The Poisson kernel can be regarded as the generating function of the reproducing kernels. We derive a 
closed formula of the latter from that of the former. 

\begin{thm}\label{thm:reprod-closed}
Let $d \ge 2$. If $\a_\k=|\k|+\f{d-2}{2} > 0$, then for $(x,t), (y,s) \in \VV_0^{d+1}$, 
\begin{align}\label{eq:PnLV0}
& \sP_n \big(\sw_{\k}; (x,t),(y,s)\big) = C_{\k} \int_{[1,1]^d} \int_0^\pi 
     L_n^{2 \a_\k} \big(t+s+2 \rho(x,t,y,s;u) \cos \t \big)  \\
       &\qquad \times e^{- \rho (x,t,y,s;u) \cos \t}  j_{\a_\k-1} \big(\rho (x,t,y,s;u) \sin \t\big) (\sin \t)^{2\a_\k-1}\d\t 
          \Phi_\k(u) \d u,  \notag
\end{align}
where $C_{\k} =  \frac{2^{\a_\k -1} \Gamma(\a_\k +\f12)}{\sqrt{\pi}} c_{\k-\f{\mathbf{1}}{2}}$ and $\rho(\cdot)$ 
is given in \eqref{eq:rho}. If $\a_\k =0$ or equivalently, $d =2$ and $|\k| =0$,  then
\begin{align}\label{eq:PnLV0_d=2}
  \sP_n \big(\sw_{\k}; (x,t),(y,s)\big)  = & \frac12 \left[ \e^{-\rho(x,t,y,s)} L_n^0\big(t+s+2\rho(x,t,y,s)\big) \right. \\
     &  \quad \left. +  \e^{\rho(x,t,y,s)}   L_n^0 \big(t+s- 2\rho(x,t,y,s)\big) \right]  \notag \\
   &  -  \f12 \rho(x,t,y,s)^2  \int_0^\pi L_n^{0} \big(t+s+2 \rho(x,t,y,s) \cos \t\big)  \notag \\
   &  \qquad  \times  e^{- \rho (x,t,y,s) \cos \t} j_1 \big(\rho (x,t,y,s) \sin \t\big) \sin \t \d\t. \notag 
\end{align}
In particular, for $d \ge 2$ and $\a_\k \ge 0$, 
\begin{equation}\label{eq:PnLV0_x=0}
 \sP_n \big(\sw_{\k}; (x,t),(0,0)\big)  = L_n^{2|\k|+d-2}(t), \qquad (x,t) \in \VV_0^{d+1}.
\end{equation}
\end{thm}

\begin{proof}
We can compare the Poisson formula for $\sQ_r(\sw_{\k}; \cdot,\cdot)$ with \eqref{eq:Mehler2} 
of the Laguerre polynomials by setting 
\begin{align*}
  & t^* = \frac12 \left(t+s - \sqrt{t^2+s^2 - 2 (x_1 y_1 u_1 +\cdots +x_d y_d u_d)} \right), \\
  & s^* =  \frac12 \left(t+s +\sqrt{t^2+s^2 - 2 (x_1 y_1 u_1 +\cdots +x_d y_d u_d)} \right),
\end{align*}
which are well defined, since the expression under the square root is nonnegative for
$\|x\| = t$ and $\|y\| =s$ by the Cauchy inequality, and they satisfy 
\begin{equation}\label{t*+s*}
   t^* + s^* = t + s \quad\hbox{and} \quad \sqrt{t^* s^*} = \rho(x,t,y,s;u). 
\end{equation}
Comparing the formula in Theorem \ref{thm:PoissonV0L} with \eqref{eq:Mehler2}, we obtain with
$\a = \a_\k$, 
$$
 \sQ_r(\sw_{\k}; x,t,y,s) =  \frac{1}{(1-r)^{\a + \f12}} c_{\k-\f{\mathbf{1}}{2}} \int_{[-1,1]^d} 
     \sum_{n=0}^\infty \frac{L_n^{\a - \f 12}(t^*)L_n^{\a - \f 12}(s^*)} 
    {L_n^{\a - \f 12}(0)} r^{n} \Phi_\k(u) \d u. 
$$
If $\a > 0$, we apply the product formula \eqref{eq:productL1} on $L_n^{\a - \f 12}(t^*)L_n^{\a - \f 12}(s^*)$
which gives an integral representation that contains only $t^*+s^*$ and $t^* s^*$, so that we obtain, 
by \eqref{t*+s*},
\begin{align*}
 \frac{L_n^{\a - \f 12}(t^*)L_n^{\a - \f 12}(s^*)} 
    {L_n^{\a - \f 12}(0)} =\,& \frac{2^{\a-\f12} \Gamma(\a+\f12)}{\sqrt{2\pi}}
      \sum_{n=0}^\infty \int_0^\pi L_n^{\a-\f12} \big(t +s + 2 \rho(x,t,y,s;u) \cos \t \big) \\
       &\times e^{- \rho(x,t,y,s;u)  \cos \t}  j_{\a-1} \left(\rho(x,t,y,s;u)\sin \t\right) (\sin \t)^{2\a-1}\d\t. \, \notag
\end{align*}
From the generating function \eqref{eq:generatingL} it follows readily that  
$$
\frac{1}{(1-r)^{\b+1}} \sum_{n=0}^\infty L_n^\a (x) r^n =  \sum_{n=0}^\infty L_n^{\a+\b+1} (x) r^n.  
$$
Consequently, putting the last three displayed formula together, we deduce that  
\begin{align*}
  \sQ_r(\sw_\k; x,t,y,s) \, & = C_\k \int_{[-1,1]^d} \int_0^\pi L_n^{2\a} \big(t +s + 2 \rho(x,t,y,s;u) \cos \t \big) r^{n}  \\
       &\times e^{- \rho(x,t,y,s;u) \cos \t}  j_{\a-1} \left(\rho(x,t,y,s;u)\sin \t\right) (\sin \t)^{2\a-1}\d \t 
        \Phi_\k(u) \d u.
\end{align*}
By \eqref{eq:PoissonLV0}, the coefficient of $r^n$ is the reproducing kernel $\sP_n(\sw_{\k}; \cdot,\cdot)$, 
which gives the stated formula for $\a > 0$. The proof for $\a=0$ follows the same procedure but 
using \eqref{eq:productL2} instead. Finally, using $\rho(x,t,0,0) =0$ and $j_\a(0) = 1/(2^\a \Gamma(\a+1))$,
\eqref{eq:PnLV0_x=0} follows readily from \eqref{eq:PnLV0} and \eqref{eq:PnLV0_d=2}. 
\end{proof}

\section{Convolution structure on conic surface with Laguerre weight}

The closed form of the reproducing kernel suggests a pseudo convolution structure on the conic surface,
which is bounded in the $L^p$ space defined on $\VV_0^{d+1}$ as follows: $f\in L^p_{\su(\k)}(\VV_0^{d+1})$ 
if $\|f\|_{p,\k}$ is finite, where 
$$
  \|f \|_{p, \k} : = \left( \ssb_\k \int_{\VV_0^{d+1}} \left |f(x,t) \e^{-t/2} \right |^p \su_\k(x,t) \d \s(x,t) \right)^{\f1 p}, 
   \quad   \su_k(x,t) =  t^{-1} h_\k^2(x),
$$
for $1 \le p < \infty$, and 
$$
  \|f\|_{\infty,\k} : = \esssup \left \{|f(x,t)| \e^{- t/2}: (x,t) \in \VV_0^{d+1} \right\}. 
$$
For $p =2$, the norm $ \|\cdot\|_{2, \k}$ coincides with the usual norm of $L^2(\VV_0^{d+1}, \sw_\k)$. 

The pseudo convolution is defined via a generalized translation operator, which is defined and shown to be
bounded in the first subsection. The convolution structure is studied in the second subsection, which is used 
to analyze the Ces\`aro summability of the generalized Laguerre series on the conic surface in the third 
subsection.

\subsection{Translation operator} 
For $d \ge 2$, $\k_i \ge 0$, let $\a_\k = |\k|+ \f{d-2}2$ as before. For $g: \RR_+ \mapsto \RR$ we define a 
generalized translation operator $\sT_{(x,t)}$. 

\begin{defn}\label{def:sT-transl}
Let $g \in L^1_{u(2\a_\k)}(\RR_+)$ and $(x,y) \in \VV_0^{d+1}$. If $\a_k > 0$, define
\begin{align*}
 \sT_{(x,t)} g(y,s) = \, & C_\k \int_{[-1,1]^d}\int_{0}^\pi g(t+s+2 \rho(x,t,y,s;u) \cos \t \big)\\
       &\times e^{- \rho (x,t,y,s;u) \cos \t}  j_{\a_\k -1} \big(\rho (x,t,y,s;u) \sin \t\big) (\sin \t)^{2\a_\k -1}\d\t
        \Phi_\k(u) \d u,  \notag
\end{align*}
where $C_\k$ is given in \eqref{eq:PnLV0}; if $d =2$ and $\k = 0$, define
\begin{align*}
  \sT_{(x,t)} g(y,s)  = \, & \frac12 \left[ \e^{-\rho(x,t,y,s)} g\big(t+s+2\rho(x,t,y,s)\big) + 
     \e^{\rho(x,t,y,s)}  g \big(t+s- 2\rho(x,t,y,s)\big) \right]  \notag \\
   &  -  \f12 \rho(x,t,y,s)^2  \int_0^\pi g \big(t+s+2 \rho(x,t,y,s) \cos \t\big)  \notag \\
   &  \qquad\qquad\qquad\quad  \times  e^{- \rho (x,t,y,s) \cos \t} j_1 \big(\rho (x,t,y,s) \sin \t\big) \sin \t \d\t. \notag
\end{align*}
\end{defn}

The definition is motivated by the closed form formula of the reproducing kernel. Indeed, it 
follows readily that 
\begin{equation} \label{eq:TLn=Pn}
    \sP_n(\sw_{\k}; (x,t),\cdot ) =  \sT_{(x,t)} L_n^{2|\k|+d-2}, \qquad n = 0, 1, 2, \ldots. 
\end{equation}
The generalized translation operator will be used to define a pseudo convolution structure on the conic surface
in the next subsection. We first show that this operator is bounded. It turns out that we need to consider
two separated cases. 

\begin{prop} \label{prop:sTbd}
Let $d \ge 2$. Assume $\a_k = |\k|+\f{d-2}{2} \ge \f12$. If $g\in L^p_{u(2\a_\k)}(\RR_+)$ for $ 1 \le p \le \infty$,
then
$$
   \| \sT_{(x,t)} g\|_{p, \k} \le  e^{t/2} \|g\|_{p,  u(2\a_\k)}, \qquad 1 \le p \le \infty.
$$
\end{prop}

\begin{proof}
Setting $x= t \xi$ and $y = s\eta$, $\xi, \eta \in \sph$, we can write 
$$
\rho(x,t,y,s;u) = \sqrt{t s} v(\xi,\eta;u), \qquad v(\xi,\eta;u) = \sqrt{\tfrac12 (1+ \xi_1 \eta_1 u_1 + \cdots  +\xi_d \eta_d u_d)}. 
$$
To simplify the notation, we write $\a = \a_\k$ throughout the proof and define 
$$
 G(t,s; z)  =  g\left(t+s + 2 \sqrt{ts}\, z\right) \e^{- \sqrt{t s}\, z}.
$$
The integral representation of the Bessel function \cite[(10.9.4)]{DLMF} states that
\begin{equation}\label{eq:jaInteg}
 j_{\a}(z) = \frac{2^{-\a}}{\sqrt{\pi}\Gamma(\a+\f12)} \int_{-1}^1 e^{\i z u} (1-u^2)^{\a-\f12}\d u, \quad \a > - \tfrac12, 
\end{equation}
and the formula holds under the limit when $\a \to -\f12+$, so that $j_{-\f12}(z) = \sqrt{\f{2}\pi} \cos z$. It 
implies, in particular, that $|j_{\a-1}(z)| \le 1/(2^{\a-1} \Gamma(\a))$ for all $z \in \RR$ and $\a \ge \f12$. 
With the explicit formula of $C_\k$, it follows that $C_\k | j_{\a-1}(z)| \le c_{\a - \f12}c_{\k - \f12}$. This last inequality 
leads to an upper bound of $|\sT_{(x,t)} g|$, 
\begin{align*}
 |\sT_{(x,t)} g(y,s)| \,& \le c_{\a-\f12}c_{\k-\f{\mathbf{1}}{2}}
    \int_{[-1,1]^d} \int_0^\pi \big |G\big(t,s; v(\xi,\eta;u)\cos\t\big) \big | (\sin\t)^{2\a-1} \d \t \Phi_\k(u) \d u \\
      & = : F(t, \xi, s, \eta).
 \end{align*} 
By the definition of $\|g\|_{\infty,\k}$, it follows readily that 
$$
    |G(t,s; u)| = \left| g\left(t+s + 2 \sqrt{ts}\, u\right)\right| \e^{- (t+s + 2 \sqrt{t s}\, u)/2} \e^{(t+s)/2}  
      \le \e^{(t+s)/2} \|g\|_{\infty, u(2\a_\k)},
$$
which leads immediately to 
$$
   \left \| \sT_{(x,t)}g \right \|_{\infty,\k} = \sup_{(y,s)}
        \left |\sT_{(x,t)}\big(t,s, v(\xi,\eta)\big) \right| \e^{- s/2}  \le \e^{t/2}\|g\|_{\infty, u(2\a_\k)}. 
$$
This establishes the stated result for $p = \infty$. 

Next we consider $p=1$. Using the intertwining operator $V_\k$, we can write
\begin{align*}
  F(t, \xi, s, \eta) =   
   c_{\a-\f12}  \int_0^\pi  V_\k \left[ \left |G \left (t,s;  \sqrt{\frac{1+\la\xi, \cdot\ra}{2}} \cos \t\right)\right| \right](\eta) 
      (\sin\t)^{2\a-1} \d \t. 
\end{align*}
Hence, using the identity \eqref{eq:intVk}, we obtain that 
\begin{align*}
  I_{(x,t)}(s):=  \, & c_\k^h  \int_{\sph}  F(t, \xi, s, \eta) h_\k^2(\eta) \d\s_\SS(\eta) \\
    =  \,&  c_{\a-\f12}  \int_0^\pi  c_{\a}  \int_{-1}^1 \left|G\left(t,s,  \sqrt{\f{1+u}{2}}\cos\t \right)\right|
     (1-u^2)^{\a - \f12} \d u  (\sin\t)^{2\a-1} \d \t  \\
   = \,& 2^{2 \a+1}   c_{\a-\f12}c_{\a} 
          \int_0^\pi \int_{0}^1 \left|G(t,s, v \cos \t)\right| v^{2 \a} (1-v^2)^{\a - \f12} \d v (\sin\t)^{2\a-1} \d \t,
\end{align*}
where the second step follows from changing variable $v\to \sqrt{(1+u)/2}$. Changing variable 
$\t\mapsto u = v\cos \t$ in the integral with respect to $\d \t$, we further obtain that
\begin{align}\label{eq:intTsph}
    I_{(x,t)}(s) \, & =   2^{2 \a+1}   c_{\a-\f12}c_{\a} 
      \int_0^1 \int_{-v}^v |G(t,s; u)| (v^2 - u^2)^{\a-1}\d u\, v  (1-v^2)^{\a - \f12} \d v  \\
  &  =   2^{2 \a+1}   c_{\a-\f12}c_{\a}   \int_{-1}^1 |G(t,s; u)|  
      \int_{|u|}^1 v (v^2 - u^2)^{\a - 1} (1-v^2)^{\a - \f12} \d v\, \d u \notag \\
  & =   2^{2 \a}   c_{\a-\f12}c_{\a} 
     \int_{-1}^1 |G(t,s; u)| (1-u^2)^{2\a - \f12} \d u   \int_0^1 z^{\a-1} (1-z)^{\a-\f{1}{2}} \d z \notag \\
  & = c_{2\a}  \int_{-1}^1 |G(t,s; u)| (1-u^2)^{2\a - \f12} \d u, \notag
\end{align}
where we have verified in the last step, using $\Gamma(2 a) = 2^{2a-1} \Gamma(a) \Gamma(a+\f12) /\sqrt{\pi}$, that 
$$
2^{2 \a} c_{\a-\f12}c_{\a} \int_0^1 z^{\a-1} (1-z)^{\a-\f{1}{2}} \d z = 
2^{2 \a} c_{\a-\f12}c_{\a}\frac{\Gamma(\a)\Gamma(\a+\frac12)}{\Gamma(2\a+\f12)} =   c_{2\a}.  
$$
Changing one more variable $u \to z = t+s - 2 \sqrt{t s} u$ in the last integral of \eqref{eq:intTsph}, we further 
obtain 
\begin{align}\label{eq:I-bound}
  I_{(x,t)}(s) \, & = c_{2\a} \int_{z_-}^{z_+} |g(z)| \e^{- \frac{z-t-s}{2}} \left(1- \frac{(z-t-s)^2}{4 ts}\right)^{2\a-\f12} \frac{\d z}{2 \sqrt{t s}}\\ 
      & = c_{2\a} \int_0^\infty |g(z)|  H_\a(t,s;z) z^{2 \a} \e^{-z} \d z,\notag
\end{align}
where $z_{\pm} = t+s \pm 2 \sqrt{ts} = (\sqrt{t} \pm \sqrt{s})^2$ and 
$$
  H_\a(t,s;z) = \frac{(t^2 + s^2 + z^2 - 2 t s - 2 t z -2 s z)^{2\a -\f{1}2}}{ ( 4 t s z )^{2\a}} \e^{\f{t+s+z}{2}},
$$
if $z \in [z_-, z_+]$ and $H_d(t,s; z) = 0$ otherwise. In particular, it follows that $H_\a(t,s,z) = 0$ if 
$2 t s+ 2 t z +2 s z \le  t^2 + s^2 +z^2$, so that $H_\a(t,s;z)$ is symmetric in $t, s$ and $z$ and 
it is nonnegative. Moreover, choosing $g(z) = \e^{(-z+t+s)/2}$ in the \eqref{eq:I-bound}, it follows readily that
\begin{align}\label{eq:IntH=1}
  c_{2\a} \int_0^\infty H_\a(t,s; z) z^{2\a} \e^{-\f{(z+t+s)} 2} \d z = c_{2\a} \int_{-1}^1  (1-u^2)^{2\a-\f12} \d u =1.
\end{align}
Consequently, exchanging the order of integrals and, by the symmetry of $H_\a(t,s;z)$, integrating 
$H_\a(t,s;z)$ with respect to $s$ first, we obtain by \eqref{eq:I-bound} and \eqref{eq:IntH=1} that
\begin{align*}
    \left \| \sT_{(x,t)}g \right \|_{1,\k}\, & =  \ssb_\k \int_{\VV_0^{d+1}} \left|\sT_{(x,t)} g(y,s)\e^{-s/2}\right| 
          s^{-1} h_\k^2(y)\d \s(y,s) \\  
    & \le  b_{2\a} \int_0^\infty   I_{(x,t)}(s)  s^{2 \a} \e^{-s/2} \d s \\
    & = c_{2\a}b_{2\a}  
    \int_0^\infty |g(z)| \int_0^\infty   H_\a(t,s;z)  s^{2\a} \e^{-(s+z)/2} \d s  z^{2\a} \e^{-z/2} \d z\\
    &  = \e^{t/2}b_{2\a} \int_0^\infty |g(z)| z^{2\a} \e^{-z/2} \d z = \e^{t/2} \|g\|_{1, u(2\a)},
\end{align*}
which completes the proof for $p =1$. The case $1 < p < \infty$ follows from the Riesz-Thorin theorem 
by interpolating between the estimates for $p=1$ and $p = \infty$. This completes the proof. 
\end{proof}

The operator $\sT_{(x,t)}$ is not a positive operator because of the presence of the Bessel function, but its 
boundedness in the $L^p_{\su(\k)}$ norm is as strong as that of a positive operator in the sense that 
it is bounded with a constant 1 in the right-hand side of the inequality in the Proposition \ref{prop:sTbd}. 

The above proposition holds under the assumption that $\a_\k \ge \f12$, which comes from the upper
bound of $j_{\a_\k -1}(t)$ deduced from \eqref{eq:jaInteg}. For $\a_\k < \f12$, we need to a different upper
bound of $j_{\a_\k-1}$ given in the following lemma. 

\begin{lem}\label{lem:j-bound}
For $\a > 0$ and $ t\in \in \RR_+$, 
$$
    \left | j_{\a-\f12}(t) \right | \le \frac{1}{2^{\a-1} \Gamma(\a)} +  \frac{t^{-1}}{2^{\a} \sqrt{\pi} \Gamma(\a+\f32)}. 
$$
\end{lem}

\begin{proof}
We start from the three-term relation $J_{\a-1}(t) + J_{\a+1}(t) = 2\a t^{-1} J_\a(t)$ of the Bessel function
\cite[(10.6.1)]{DLMF}, which gives immediately
$$
       j_{\a-1}(t) = 2 \a j_\a(t) - t^2 j_{\a+1}(t), \qquad \a > 0.  
$$
We use \eqref{eq:jaInteg} for $j_{\a+1}$ and integrate by parts once to obtain 
\begin{align*}
   j_{\a+1}(t) =  \frac{2^{-\a-1}}{\sqrt{\pi}\Gamma(\a+\f32)} \frac{2\a+1}{i t}  \int_{-1}^1 \e^{\i t u} u (1-u^2)^{\a-\f12}\d u
\end{align*}
which implies that $ t \left | j_{\a+1}(t) \right | \le  \frac{2^{-\a}}{\sqrt{\pi}\Gamma(\a+\f32)}$. Hence, 
using $2\a |j_\a(t)| \le \frac{1}{2^{\a-1} \Gamma(\a)}$, the desired estimate follows from the three term relation.
\end{proof}

\begin{prop} \label{prop:sTbd2}
Let $d \ge 2$ and $\a_\k= |\k| + \f{d-2}{2} < \f12$. For $g\in L^1_{u(2\a_\k)}(\RR_+)\cap L^1_{u(2\a_\k+\f12)}(\RR_+)$, 
$$
   \| \sT_{(x,t)} g\|_{1, \k} \le  e^{t/2} \|g\|_{1,u(2\a_k)} + \frac{1}{(2\a_\k+1)\sqrt{\pi}} \sqrt{t} \e^{t/2} \|g\|_{1, u(2\a_\k+\f12)}.
$$
\end{prop}

\begin{proof}
Let $\a = \a_k$. We need to consider two cases, $\a =0$ and $\a >0$. Assume first $\a > 0$. We adopt 
the same notation as in the proof of the previous proposition for $\a \ge \f12$. Using  Lemma \ref{lem:j-bound}, 
we obtain $C_\k |j_{\a-\f12}(z)| \le c_{\a-\f12} c_{\k-\f{\mathbf{1}}2}+ \frac{z}{(2\a+1)\pi}c_{\k-\f{\mathbf{1}}2}$. 
Applying this inequality with $z = \rho(x,t,y,s;u)= \sqrt{t s} v (\xi,\eta;u)$,  we obtain 
\begin{align*}
  \left | \sT_{(x,t)} g(y,s) \right| \le F(t, \xi, s, \eta)+ F_1(t, \xi, s, \eta),
\end{align*}
where $F(t, \xi, s, \eta)$ is the same as before and $F_1(t, \xi, s, \eta)$ is defined by 
\begin{align}\label{eq:F1}
F_1(t, \xi, s, \eta) =  \frac{\sqrt{t s}}{(2\a+1)\pi} \int_0^\pi  V_\k & \left[ \sqrt{\frac{1+\la\xi, \cdot\ra}{2}} \left |G \left (t,s;  \sqrt{\frac{1+\la\xi, \cdot\ra}{2}} \cos \t\right)\right| \right](\eta) \notag  \\
    &  \times (\sin\t)^{2\a} \d \t  
\end{align}
in terms of the intertwining operator $V_\k$. The estimate of $F(t,\xi,s, \eta)$ for $\a \ge \f12$ remains valid 
for $\a < \f12$, which gives the first term in the right-hand side of the desired estimate.  The estimate of 
$F_1(t,\xi,s,\eta)$ can be carried out by the similar approach. In particular, following the proof of Proposition
\ref{prop:sTbd} and checking the constant carefully, we obtain 
\begin{align*} 
I_{(x,t)} (s) := \, &  c_\k^h \int_{\SS^{d-1}} F_1(t,\xi,s,\eta)  h_\k^2(\eta) d\s_{\SS}(\eta) \\
                   =  \,&  \frac{\sqrt{t s}}{(2\a+1)\pi} \int_{-1}^1 |G(t,s;u)| (1-u^2)^{2\a} \d u. \notag
\end{align*}
Changing variable $u \to z = t+s - 2 \sqrt{t s} u$ as in the case of $\a_\k \ge \f12$, we further obtain
\begin{align*}
  I_{(x,t)}(s) = \frac{1}{(2\a+1)\pi} \int_0^\infty |g(z)|  H_\a^{(1)}(t,s;z) z^{2 \a} \e^{-z} \d z,\notag
\end{align*}
where $H_\a^{(1)}$ is defined by 
$$
  H_\a^{(1)}(t,s;z) = \frac{(t^2 + s^2 + z^2 - 2 t s - 2 t z -2 s z)^{2\a }}{ 2 ( 4 t s z )^{2\a}} \e^{\f{t+s+z}{2}},
$$
if $z \in [z_-, z_+]$ and $H_d(t,s; z) = 0$ otherwise. Again, $H_\a^{(1)}(t,s;z)$ is symmetric in $t, s$ and $z$ and 
it is nonnegative. Furthermore, we also have
\begin{align*} 
    \int_0^\infty H_\a(t,s; z)^{(1)} z^{2\a} \e^{-\f{(z+t+s)} 2} \d z = \sqrt{ts} \int_{-1}^1  (1-u^2)^{2\a} \d u = 
       \frac{\sqrt{ts}}{c_{2\a+\f12}}.
\end{align*} 
Consequently, continuing as in the proof of Proposition \ref{prop:sTbd}, we conclude that 
\begin{align*}
  \| F_1(t,\xi, \cdot) \|_{1,\k}\, & =  b_{2\a} \int_0^\infty  I_{(x,t)}(s)  s^{2 \a} \e^{-s/2} \d s \\
    & =  \frac{ b_{2\a}}{(2\a+1)\pi}   
    \int_0^\infty |g(z)| \int_0^\infty   H_\a^{(1)}(t,s;z)  s^{2\a} \e^{-(s+z)/2} \d s  z^{2\a} \e^{-z/2} \d z\\
    &  =\sqrt{t}\e^{t/2} \frac{ b_{2\a}}{(2\a+1)\pi c_{\a+\f12}} \int_0^\infty |g(z)| z^{2\a+\f12} 
         \e^{-z/2} \d z \\
    & = \frac{1}{(2\a+1)\sqrt{\pi}} \sqrt{t} \e^{t/2} \|g\|_{1, u(2\a+\f12)},
\end{align*}
where we have used $b_{2\a} = 1/\Gamma(2\a+1)$ and $1/c_{2\a+\f12} = b_{2\a+\f12}\sqrt{\pi}/b_{2\a}$
in the last step. This completes the proof for $\a > 0$.  

Next we consider $\a =0$, which is equivalent to $\k = 0$ and $d =2$ since $|\k| \ge 0$ and $d \ge 2$. This
corresponds to the Laguerre weight $\sw_0(x,t) = t^{-1} \e^{-t}$ on the conic surface $\VV_0^3$ of $\RR^3$. 
In this case, we write the operator $\sT_{(x,t)}$ as a sum of two parts,
$$
\sT_{(x,t)} =  \sT_{(x,t)}^{(1)}+  \sT_{(x,t)}^{(2)},
$$
where, using the notation $G(t,s;u)$ and $v(\xi,\eta) = \sqrt{\f12(1+\la \xi,\eta\ra)}$, we have 
\begin{align*}
  \sT_{(x,t)}^{(1)} g(y,s) \, & =  \frac 12 \left[G(t,s; v(\xi,\eta)) +  G(t,s; - v(\xi,\eta))\right]; \\
  \sT_{(x,t)}^{(2)} g(y,s)  \,& = - \f{t s}{2} [v(\xi,\eta)]^2 \int_0^\pi G(t,s; v(\xi,\eta) \cos\t) 
               j_1 \big(\sqrt{ts} v(\xi,\eta) \sin \t\big) \sin \t \d\t.
\end{align*}
For $\sT_{(x,t)}^{(1)}g$, we first integrate over $\SS^1$, which is the unit circle and $\ssb_2 = 2\pi$. 
Let $\la \xi, \eta \ra = \cos \phi$. Then $v(\xi,\eta) = \cos \frac{\phi}{2}$. Hence, 
\begin{align*}
  \frac{1}{2\pi} \int_{\SS^1} \left|  \sT_{(x,t)}^{(1)} g(y,s) \right| d\s \, & =  \frac{1}{4\pi} \int_0^{2 \pi}
      \left | G\left(t,s; \cos \f{\phi}2\right) +  G\left (t,s; - \cos \f{\phi}2\right) \right | \d \phi \\
      & = \frac{1}{2\pi} \int_{-1}^1 \left | G(t,s; u) +  G(t,s; - u) \right | \frac{\d u}{\sqrt{1-u^2}} \\
       & \le \frac{1}{\pi} \int_{-1}^1 | G(t,s; u)| \frac{\d u}{\sqrt{1-u^2}}. 
\end{align*}
Sine $c_0 = 1/\pi$, we see that the integral in the right-hand side is the integral in \eqref{eq:intTsph} 
when $\a_k =0$. The proof below \eqref{eq:intTsph} for $\a_\k \ge \f12$ remains valid for $\a_\k = 0$, 
which proves the bound for $\|\sT_{(x,t)}^{(1)} g\|_{1,0} \le  e^{t/2} \|g\|_{1,u(0)}$. For $\sT_{(x,t)}^{(2)}$, 
we use the bound $ | z j_1(z)| \le \frac2 \pi$ in the proof of Lemma \ref{lem:j-bound} to obtain 
$$
  \left|\sT_{(x,t)}^{(2)} g(y,s) \right| \le \f 1 \pi \sqrt{t s} |v(\xi,\eta)| \int_{0}^\pi \left|G(t,s; v(\xi,\eta) \cos\t)\right| \d \t.
 $$
Since $V_\k$ becomes identity operator when $\k = 0$, the right-hand side of the above expression 
coincides with $F_1(t,\xi,s,\eta)$ with $\a =0$. The evaluation of the norm of $F_1$ in the case of $\a > 0$
remains valid for $\a =0$, so that the norm $\|\sT_{(x,t)}^{(2)} g\|_{1,0} \le \f1{\sqrt{\pi}} \|g\|_{1,u(\f12)}$ 
follows right away. This completes the proof. 
\end{proof}
 
\subsection{Pseudo convolution structure on the conic surface}
We define a pseudo convolution structure on the conic surface. 

\begin{defn}\label{def:f*g}
Let $d \ge 2$, $\k_i \ge 0$ and $\a_\k = |\k|+\f{d-2}2$. For $f \in L^1_{\su(\k)}(\VV_0^{d+1})$ and
$g\in L^1_{u(2\a_\k)}(\RR_+)$, define
$$
   f* g(x,t) = \ssb_{\k} \int_{\VV_0^{d+1}} f(y,s) \sT_{(x,t)} g(y,s) \sw_{\k} (y,s) d\s(y,s), \quad (x,t) \in \VV_0^{d+1}.
$$
\end{defn}

From \eqref{eq:TLn=Pn} it follows readily that, for $f \in L^2_{\su(\k)}(\VV_0^{d+1})$, 
\begin{equation}\label{eq:f*Ln}
  \proj_n (\sw_{\k}; f) = f* L_n^{2|\k|+d-2}, \qquad n = 0, 1, 2, \ldots. 
\end{equation}
The following proposition justifies calling it a convolution.  

\begin{prop}\label{prop:conv-coeff}
Let $d\ge 2$. For $f\in L^2(\VV_0^{d+1}, \sw_\k)$ and $g \in L^2(\RR_+, w_{2\a_\k})$,  
\begin{equation}\label{eq:conv-coeff}
      \proj_n \big(\sw_\k; f*g\big) = \wh g_n^{2\a_\k} \proj_n \big(\sw_{\k}; f\big),
\end{equation}
where $\wh g_n^{2\a_\k}$ is the Fourier-Laguerre coefficient as defined in \eqref{eq:FourierLaguerre}.
In particular, 
\begin{equation}\label{eq:f*gL2}
         \|f*g\|_{2,k} \le \|f\|_{2,\k} \|g\|_{2,u(2\a_\k)}. 
\end{equation}
\end{prop}

\begin{proof}
Writing $g$ in terms of its Laguerre expansion in $L^2(\RR_+, w_{2\a_\k})$ as in \eqref{eq:FourierLaguerre}, 
it follows by \eqref{eq:TLn=Pn} that 
$$
  \sT_{(x,t)} g 
        =  \sum_{n=0}^\infty  \wh g_n^{2\a_\k} \sP_n((x,t), \cdot) \quad \hbox{and}\quad
      f*g =    \sum_{n=0}^\infty  \wh g_n^{2\a_\k} \proj_n(\sw_\k),
$$
where the second identity follows from the first one by the definition of $f*g$. Since $\|\cdot\|_{2,\k}$ 
coincides with the norm of $L^2(\VV_0^{d+1}, \sw_\k)$ and $\|\cdot\|_{2,u(\a)}$ coincides with the 
norm of $L^2(\RR_+, w_\a)$, we have the Parseval identities 
$$
  \|f*g\|_{2,k}^2 = \sum_{n=0}^\infty \left \| \proj_n \big(\sw_\k; f*g\big) \right \|_{2,\k}^2  \qquad 
  \hbox{and}\qquad 
   \|g\|_{2,u(\a_\k)}^2 = \sum_{n=0}^\infty  \left| \wh g_n^{2\a_\k} \right|^2. 
$$
Hence, \eqref{eq:f*gL2} follows from \eqref{eq:conv-coeff} and the Cauchy inequality. 
\end{proof}

We now show that the pseudo convolution operator is bounded in $L^p_{\su(\k)}(\VV_0^{d+1})$ for $p \ne 2$.
We again need to consider two cases: $\a_\k \ge \f12$ and $\a_k < \f12$. 

\begin{thm} \label{thm:f*g-bdd}
Let $d \ge 2$. Assume $\a_\k = |\k|+\f{d-2}2 \ge \f12$. For $f \in L^p_{\su(\k)}(\VV_0^{d+1})$, $1 \le p\le \infty$, 
and $g\in L^1_{u(2\a_\k)}(\RR_+)$, 
$$
  \|f*g\|_{p, \k} \le \|f\|_{p,\k} \|g\|_{1, u(2\a_\k)}, \qquad 1 \le p \le \infty. 
$$
\end{thm}

\begin{proof} 
If $ p =1$, then using the Proposition \ref{prop:sTbd} and the obvious symmetry of $\sT_{(x,t)} (y,s)$ in $(x,t)$ and $(y,s)$, 
we obtain
\begin{align*}
  \|f*g\|_{1,\k} \,& = \ssb_\k \int_{\VV_0^{d+1}} \left| \ssb_\k \int_{\VV_0^{d+1}} f(y,s) \sT_{(x,t)}g(y,s)
        \sw_\k(y,s)\d \s(y,s)\e^{-t/2} \right| \su_\k(x,t) \d \s(x,t) \\
     & \le \ssb_\k \int_{\VV_0^{d+1}} | f(y,s)| \cdot \|\sT_{(y,s)}g \|_{\k,1}\sw_\k(y,s)\d \s(y,s)\\
     &  \le  \|g\|_{1,u(2\a_\k)}  \ssb_\k \int_{\VV_0^{d+1}} | f(y,s)| e^{s /2}  \sw_\k(y,s) \d \s(y,s)  \\
     & = \|f\|_{1,\k} \|g\|_{1, u(2\a_\k)}, 
\end{align*}
which proves the stated inequality for $p=1$. Furthermore, for $p = \infty$, 
\begin{align*}
  \|f*g\|_{\infty,\k} \,& = \sup_{(x,t) \in \VV_0^{d+1}} \left | f*g(x,t) \e^{-t/2} \right |  \\
   & \le \|f\|_{\infty, \k} \sup_{(x,t) \in \VV_0^{d+1}} \ssb_\k \int_{\VV_0^{d+1}} 
       \left| \sT_{(x,t)}g(y,s)\right| \e^{s/2} \sw_\k(y,s) \d \s(y,s) \e^{-t/2}  \\
   &= \|f\|_{\infty, \k} \sup_{(x,t) \in \VV_0^{d+1}} \| \sT_{(x,t)}g\|_{\k,1} \e^{-t/2} 
     \le \|f\|_{\infty,\k} \|g\|_{1, u(2\a_\k)}, 
\end{align*}
which is the stated inequality for $p = \infty$. The case $1 < p < \infty$ follows from the Riesz-Thorin 
interpolation theorem. 
\end{proof}

A straight forward generalization of the theorem is the Young's inequality. 

\begin{cor}
Let $d \ge 2$ and $\a_\k \ge \f12$. For $1 \le p, q, r \le \infty$ with $r^{-1} = p^{-1}+q^{-1} -1$, if
$f\in L^p_{\su(\k)}(\VV_0^{d+1})$ and $g \in L^q_{u(2\a_\k)}(\RR_+)$, then $f*g \in L^r_{\su(\k)}(\VV_0^{d+1})$ and
$$
  \|f*g\|_{r, \k} \le \|f\|_{p,\k} \|g\|_{q, u(2\a_\k)}.
$$ 
\end{cor}

For the case $\a_\k \ge \f12$, the boundedness of $f*g$ is more restricted. We need to introduce another 
space $L^p_{\su(\k),*}(\VV_0^{d+1})$ that has the norm 
$$
  \|f\|_{p,\k}^* = \left( \ssb_\k \int_{\VV_0^{d+1}} |f(x,t)e^{-t/2}|^p  \sqrt{t} \,\su_\k(x,t)\d \s (x,t) \right)^{\f1p}, 
     \quad 1 \le p < \infty,
$$
which has an additional $\sqrt{t}$ in the integral in comparing with $\|f\|_{p,\k}$, and also
$$
  \|f\|_{\infty, \k}^* =  \esssup \left \{|f(x,t)| \sqrt{t} \e^{- t/2}: (x,t) \in \VV_0^{d+1} \right\}.
$$

\begin{thm}\label{thm:f*g-bdd2}
Let $d \ge 2$ and $\a_\k = |\k|+\f{d-2}2 \le \f12$. For $f \in L^1_{\su(\k)} (\VV_0^{d+1}) \cap L^1_{\su(\k),*}(\VV_0^{d+1})$ 
and $g\in L^1_{u(2\a_\k)}(\RR_+)\cap L^1_{u(2\a_\k+\f12)}(\RR_+)$, 
$$
  \|f*g\|_{1, \k} \le \|f\|_{1,\k} \|g\|_{1, u(2 \a_\k)}+  \|f\|_{1, \k}^* \|g\|_{1, u(2 \a_\k+\f12)}.
$$
\end{thm}

\begin{proof}
The proof follows along the line of that Theorem \ref{thm:f*g-bdd} using Proposition \ref{prop:sTbd} instead. 
It follows then that 
\begin{align*}
  \|f*g\|_{1,\k} \, & \le   \ssb_\k \int_{\VV_0^{d+1}} |f(y,s)| \cdot \| \sT_{(x,t)} g(y,s)\|_{1,\k} \sw_\k(y,s) \d \s(y,s) \\
      &  \le  \|g\|_{1, u(2\a_\k)}   \ssb_\k  \int_{\VV_0^{d+1}} | f(y,s)|  \su_\k(y,s) \e^{-s/2}\d \s(y,s) \\
       & \qquad
        +\frac{1}{(2\a_k+1)\pi} \|g\|_{1, u(2\a_\k+\f12)}   \ssb_\k 
            \int_{\VV_0^{d+1}} | f(y,s)| \sqrt{s} \su_\k(y,s) \e^{-s/2}\d \s(y,s)  \\
      & = \|f\|_{1, \k} \|g\|_{1,u(2\a_k)} +\frac{1}{(2\a_k+1)\sqrt{\pi}} \|f\|_{1,\k}^* \|g\|_{1,u(2\a_k+\f12)}.  
 \end{align*}
Since $\a_\k \ge 0$, we can replace the constant by 1. This proves the stated inequality. 
\end{proof}

We restrict to the case $p=1$ in the above theorem. For $ p = \infty$, we need to modify the the norm in
the left-hand side. For example,  the following inequality
$$
 \left | f*g(x,t) \min\{1, t^{- \f12}\} \e^{-t/2} \right | \le \|f\|_{\infty,\k} \left (\|g\|_{\infty,u(2\a_k)} + \|g\|_{\infty,u(2\a_k+1)}\right).
$$
follows readily. We omit the details.

\subsection{Ces\`aro means on the conic surface}
For $\delta > -1$, the Ces\`aro means of the Fourier orthogonal expansions with respect to $\sw_\k$ is given by
\begin{align*}
  \sS_n^\delta(\sw_\k; f) \, & = \frac{1}{\binom{n+\delta}{n}} \sum_{k=0}^n  \binom{n-k+\delta}{n-k} \proj_k (\sw_\k; f) \\
       & = \ssb_\k \int_{\VV_0^{d+1}} f(y,s) \sK_n^\delta(\sw_\k; (x,t),(y,s)) \sw_\k(y,s) \d \s (y,s),
\end{align*}
where $\sK_n^\delta(\sw_\k)$ denote the $(C,\delta)$ kernel 
$$
\sK_n^\delta(\sw_\k;(x,t),(y,s)) =  \frac{1}{\binom{n+\delta}{n}} \sum_{k=0}^n  \binom{n-k+\delta}{n-k} 
   \sP_k(\sw_\k; (x,t),(y,s)).
$$
By the closed formula of the projection operator in \eqref{eq:f*Ln}, we can express the means in terms of the
puedo convolution. 

\begin{prop} \label{prop:CesaroKV0}
Let $d \ge 2$ and $\delta > -1$. Then 
\begin{equation}\label{eq:Snd-sTd}
  \sK_n^\delta \big(\sw_{\k}; (x,t), (y,s)\big) = \frac{1}{\binom{n+\delta}{n}} \sT_{(x,t)}  \left(L_n^{\delta+2|\k|+ d-1}\right) (y,s), 
\end{equation}
and
\begin{equation}\label{eq:Snd-convo}
  \sS_n^\delta (\sw_\k; f) =  \frac{1}{\binom{n+\delta}{n}} f* L_n^{\delta+2|\k|+d-1}.
\end{equation}
\end{prop}

\begin{proof}
By \eqref{eq:TLn=Pn}, we see that 
$$
  \sK_n(\sw_\k; (x,t),\cdot) = \sT_{(x,t)} \tau_n^\delta \quad\hbox{with}\quad \tau_n^\delta =  \frac{1}{\binom{n+\delta}{n}}
        \sum_{k=0}^n  \binom{n-k+\delta}{n-k}L_k^{2|\k|+d-2}.
$$
From the identity \eqref{eq:Cesaro} and the generating function \eqref{eq:generatingL}, it follows that 
\begin{align*}
   \sum_{n=0}^\infty \binom{n+\delta}{n} \tau_n^\delta(z) \, & = \frac{1}{(1-r)^{\delta+1}} \sum_{n=0}^\infty 
       L_n^{2|\k|+d-2}(z) r^n \\ 
       & = \frac{1}{(1-r)^{\delta+2|\k|+d}} \e^{ - \frac{zr}{1-r}}  = \sum_{n=0}^\infty  L_n^{\delta+ 2|\k|+d-1}(z) r^n. 
\end{align*}
Hence, comparing the coefficients of $r^n$, we obtain
$$
     \tau_n^\delta = \frac{1}{\binom{n+\delta}{n}} L_n^{\delta+2|\k|+d-1},
$$
which proves \eqref{eq:Snd-sTd} and, consequently, \eqref{eq:Snd-convo}. 
\end{proof}

For $1 \le p \le \infty$, let $\|\sS_n^\delta(\sw_\k)\|_{p, \k}$ denote the operator norm of $\sS_n^\delta$ defined by
$$
 \|\sS_n^\delta(\sw_\k)\|_{p, \k} =  \sup \left \{\| \sS_n^\delta (\sw_\k; f) \|_{p,\k}:  f\in L^p_{\su(\k)}(\VV_0^{d+1})\right \}.
$$

\begin{thm} \label{thm:CesaroV0}
Let $d \ge 2$ and $\delta > -1$. If $\a_\k = |\k| + \f{d-2}{2} \ge \f12$, then for $p =1$ and $p = \infty$,  
\begin{equation} \label{eq:sSn-norm}
  \|\sS_n^\delta (\sw_\k)\|_{p, \k} =\frac{1}{\binom{n+\delta}{n}} \|L_n^{\delta + 2 \a_k+1}\|_{1, u(2\a_\k)}
      \sim \begin{cases} n^{2\a_\k + \f12 - \delta} &  -1< \delta < 2\a_\k +\f12 \\
       \log n & \delta = 2\a_\k +\f12 \\
        1  & \delta > 2\a_\k +\f12.  \end{cases}
\end{equation}
Moreover, $\sS_n^\delta f$ converse to $f$ in $L^p_{\su(\k)}(\VV_0^{d+1})$, $1\le p \le \infty$, if 
$\delta > 2\a_\k +\f12 = 2|\k|+d - \f32$ and the condition is sharp if $p =1$ and $p = \infty$. 
\end{thm}

\begin{proof}
Let $\a_\k = |\k| + \f{d-2}{2}$. We need to consider two cases $\a_\k > 0$ and $\a_\k =0$ separately.  
Since $\sS_n^\delta(\sw_{\k}; f)$ is an integral operator, a standard argument shows that,
\begin{equation}\label{eq:SnL1norm}
 \|\sS_n^\delta(\sw_{\k}) \|_{p, \k}  = 
  \sup_{(x,t) \in \VV_0^{d+1}}  \int_{\VV_0^{d+1}} \left| \sK_n^\delta \big(\sw_\k; (x,t), (y,s)\big)\right| 
      \e^{s/2} \sw_\k(y,s)\d \s(y,s)  
\end{equation}
for $p=1$ and $p= \infty$. By the estimate in Proposition \ref{prop:sTbd} or Theorem \ref{thm:f*g-bdd}, we
have
$$
 \|\sS_n^\delta(\sw_\k) \|_{p, \k} \le \frac{1}{\binom{n+\delta}{n}} \|L_n^{\delta + 2 \a_k +1}\|_{1, u(2\a_\k)}. 
$$
Furthermore, by \eqref{eq:PnLV0_x=0}, we see that $\sK_n^\delta \big(\sw_{\k}; (0,0), (y,s)\big)
 =\frac{1}{\binom{n+\delta}{n}}  L_n^{\delta+ 2 \a_k +1}(t)$,
so that 
\begin{align*}
 \|\sS_n^\delta(\sw_\k) \|_{p, \k} \, & \ge \ssb_\k
     \int_{\VV_0^{d+1}} \left| \sK_n^\delta \big(\sw_\k; (0,0), (y,s)\big)\right| \sw_\k(y,s) \e^{s/2}\d \s(y,s)\\
   & = \frac{1}{\binom{n+\delta}{n}}  \int_0^\infty \left| L_n^{\delta+ 2 \a_k +1}(t)\right| t^{2\a_\k} \e^{-t/2}  \d t.
\end{align*}
Together, we have proved the equality in \eqref{eq:sSn-norm}, whereas the asymptotic in \eqref{eq:sSn-norm}
follows form Lemma \ref{lem:estLn} with $\a = 4\a_k$ and $\b = \delta - 2\a_\k+1$ and 
$\binom{n+\delta}{n} \sim n^\delta$. Consequently, it follows that $ \|\sS_n^\delta f(\sw_{\k}) \|_{p, \k}$
is bounded if $\delta > 2\a_\k +\f12$ and it is unbounded if $\delta \le 2\a_\k +\f12$. This completes the proof.
\end{proof}

It is evident that the lower bound of $\|\sS_n^\delta(\sw_\k)\|_{p,\sw_{\k}}$ in \eqref{eq:SnL1norm} remains
valid when $\a_\k < \f12$. However, the bound in Propostion \ref{prop:sTbd2} or Theorem \ref{thm:f*g-bdd2} 
does not lead to a matching upper bound.  

\begin{thm} \label{thm:Cesaro2}
Let $d \ge 2$ and $\a_\k\le \f12$. Then $\sS_n^\delta(\sw_0; f)$ converse to $f$ in $L^1_{\su(\k)}(\VV_0^{d+1})\cap
L^1_{\su(\k),*}(\VV_0^{d+1})$ if $\delta > 2 \a_\k+1$. Moreover, $\sS_n^\delta(\sw_0; f)$ does
not converge for all $f \in L^1_{\su(\k)}(\VV_0^{d+1})$ if $\delta \le 2\a_\k+\f12$. 
\end{thm}

\begin{proof}
Using  Theorem \ref{thm:f*g-bdd2}, we obtain that 
$$
 \|\sS_n^\delta(\sw_{0}; f) \|_{1, \k} \le \frac{1}{\binom{n+\delta}{n}} 
         \left(\left \|L_n^{\delta+2\a_\k+ 1}\right \|_{1, u(2\a_\k)}  \|f\|_{1, \k} +  
           \left \|L_n^{\delta+2\a_\k+ 1}\right \|_{1,u(2\a_\k+\f12)}   \|f\|_{1,\k}^* \right).
$$ 
The term $\frac{1}{\binom{n+\delta}{n}} \|L_n^{\delta+2\a_\k+ 1}\|_{1, u(2\a_\k)}$ is bounded as shown in the 
case of $\a_\k \ge \f12$. By Lemma \ref{lem:estLn} with $\a = 4\a_k+1$ and $\b = \delta - 2 \a_k1$, we obtain
$$
\frac{1}{\binom{n+\delta}{n}} \|L_n^{\delta +1}\|_{1, u(2\a_k+\f12)}
      \sim \begin{cases} n^{2\a_\k+1 - \delta} &  \delta < 2 \a_k + \f32 \\
       n^{2\a_\k+1 - \delta} \log n & \delta = 2 \a_k + \f32\\
        n^{-\f12}  & \delta > 2 \a_k + \f32,  \end{cases}
$$
which is bounded if $\delta > 2 \a_k+1$. This proves the convergence part. 
\end{proof} 

If $\k = 0$, then our result holds for $\sw_0(x,t) = t^{-1}\e^{-t}$. In this case, our result shows that
the $(C,\delta)$ means converge in $L^p_{\su(\k)}(\VV_0^{d+1})$, $p=1$ and $\infty$ if and only if 
$\delta > d - \f32$ when $d \ge 3$. For $ d =2$, the condition $\delta > \f12$ remains necessary 
but our sufficient condition $\delta > 1$ is weaker.  We believe that the condition $\delta > \f12$
should be sufficient as well, for which one may need to find a way to estimate the Ces\`aro kernel 
without using the closed form formula. 

More generally, we believe that Theorem \ref{thm:CesaroV0} should hold for all $\a_\k \ge 0$, which 
means that $\delta > 2\a_\k+\f12$ is necessary and sufficient. This is non-trivial even for the classical
Laguerre expansions in $\RR_+$, where the convergence of the $(C,\delta)$ means for $w_\a$ is 
determined via the convolution structure when $\a \ge 0$, but requires delicate hard estimate when
$\a < 0$. The same obstacle appears in the Laguerre expansions in $\RR_+^d$ with the weight 
$w_{\a_1}(x_1) \cdots w_{\a_d}(x_d)$, for which the convergence of the $(C,\delta)$ means is 
determined under the restriction $\a_i \ge 0$ \cite{Th, X00} and, as far as we are aware, the case
of negative $\a_i$ remains open. 

The product Laguerre expansions of $\RR_+^2$ is particularly pertinent to our $\sw_0$ weight when 
$d =2$. Indeed, as we shall see in the next section, the orthogonal expansion for $\sw_0$ and $d =2$ 
is closely related to the orthogonal expansion on the solid cone $\VV^2 =\{(x,t) \in \RR^2: |x| \le t\}$ with
respect to the weight function $(t^2-x^2)^{-\f12}$, and the latter is equivalent to the product Laguerre 
expansions with respect to the weight function $w_{-\f12}(x_1)w_{-\f12}(x_2)$ on $\RR_+^2$; see 
Remark \ref{rem:5.1}. 

\section{Laguerre expansions on the solid cone}
\setcounter{equation}{0}

In this section we consider the Laguerre expansion on the solid cone
$$
  \VV^{d+1} =\left \{(x,t) \in \RR^{d+1}: \|x\| \le t, \, x \in \RR^d, t \in \RR_+ \right\}, \qquad d \ge 1
$$
with respect to the weight function $W_{\k,\mu}$ defined by 
$$
 W_{\k, \mu}(x,t) = h_\k^2(x) (t^2-\|x\|^2)^{\mu-\f12} \e^{-t}, \quad \k_i \ge 0, \quad \mu > -\tfrac12. 
$$

\subsection{Orthogonal polynomials}
We define the inner product on $L^2(\VV^{d+1},W_{\k,\mu})$ with respect to $W_{\k,\mu}$ by 
\begin{align*}
    \la f,g\ra_{\k,\mu} &=  \bb_{k,\mu} \int_{\VV^{d+1}} f(x,t) g(x,t)  W_{\k,\mu}(x,t) \d x \d t, 
\end{align*}
where $\bb_{\k,\mu} = b_{2|\k|+2\mu+d-1} b_{\k,\mu}^\BB$ with $b_\mu^\BB$ being the normalization 
constant of the weight $h_\k^2(x) (1-\|x\|^2)^{\mu-\f12}$ on $\BB^d$. Let $\CV_n(\VV^{d+1},W_{\k,\mu})$ denote 
the space of orthogonal polynomials of degree $n$ with respect to this inner product. Then 
$$
   \dim \CV_n(\VV^{d+1},W_{\mu})  = \binom{n+d}{n}. 
$$
An orthogonal basis for this space can be given in terms of orthogonal polynomials on the unit ball 
and the Laguerre polynomials. 
 
\begin{prop}\label{prop:OPV}
Let $\{P_{\kb}^m(\varpi_{\k,\mu}): |\kb|=n\}$ be an orthonormal basis of $\CV_n(\BB^d,\varpi_{\k,\mu})$. Define 
\begin{equation}\label{eq:coneLbf}
   \Lb_{m,\kb}^{n}(x,t) = L_{n-m}^{2m +2|\k| + 2\mu + d-1}(t)t^m  P_{\kb}^m\left(\varpi_{\k,\mu}; \f{x} t \right ), 
     \quad |\kb| =m, \quad  0 \le m \le n.
\end{equation}
Then $\{\Lb_{m,\kb}^n: |\kb| = m, \quad 0 \le m \le n\}$ is an orthogonal basis of $\CV_n(\VV^{d+1},W_{\k,\mu})$.
Moreover, the norm square of $L_{m,\kb}^{n}$ is given by
$$
\hb^\Lb_{m,n}: = \la \Lb_{m,\kb}^n, \Lb_{m,\kb}^n\ra_{\k,\mu} =  \frac{(2|\k|+2 \mu + d)_{n+m}}{(n-m)!}.
$$
\end{prop}

The orthogonality is stated in \cite[Proposition 3.3]{X20a} and the norm can be verified directly as in the proof of
Proposition \ref{prop:OPV}. When $\k =0$, the polynomials $\Lb_{m,\ell}^n$ are called the Laguerre polynomials 
on the cone. 

\begin{rem}\label{rem:5.1}
When $d =1$, the cone $\VV_0^2$ is the wedge domain bounded by the two lines $x = \pm t$ in $\RR^2$
and the weight function is $W_{\k,\mu}(x,t) = |x|^{2\k}(t^2-x^2)^{\mu-\f12} \e^{-t}$, where we have written 
$\k = \k_1 \ge0$. If we rotate the domain by $90^\circ$ by setting $x_1  = \frac{t+x}{2}$ and $x_2 = \frac{t-x}{2}$,
then the wedge domain becomes $\RR_+^2$ and the weight function $W_{\k,\mu}$ becomes 
\begin{equation} \label{eq:U-weight}
    U_{\k,\mu} (x_1,x_2) =  |x_1-x_2|^{2\k} |x_1 x_2|^{\mu-\f12} \e^{-x_1 - x_2}, \qquad (x_1,x_2) \in \RR_+^2.
\end{equation}
In particular, for $\k = 0$, $U_{\k,\mu}(x_1,x_2) = w_{\mu-\f12}(x_1)w_{\mu-\f12}(x_2)$ is the product Laguerre
weight. 
\end{rem}

For $f \in L^2(\VV^{d+1}, W_{\k,\mu})$, the Fourier-Laguerre expansion on the cone is defined by
$$
   f = \sum_{n=0}^\infty \sum_{m=0}^n \sum_{|\kb|=m} \wh f_{m,\kb}^n \Lb_{m,\kb}^n 
   \qquad \hbox{with}\qquad  \wh f_{m,\kb}^n =  \frac{\la f, \Lb_{m,\kb}^n\ra_{\mu}} {\hb_{m,n}^\Lb}
$$
For $\k = 0$, this orthogonal expansion can be used to derive an explicit solution for the non-homogeneous
wave equation; see \cite{OX}. The projection operator $\proj_n: L^2(\VV^{d+1}; W_{\k,\mu}) \mapsto \CV_n^d(\VV^{d+1},W_{\k,\mu})$ and
the $n$-th partial sum operator $\Sb_n f$ of this expansion are defined by 
$$
    \proj_n \left(W_{\k,\mu}; f\right) = \sum_{m=0}^n \sum_{|\kb| =m}  \wh f_{m,\kb}^n \Lb_{m,\kb}^n
    \quad \hbox{and} \quad  \Sb_n f\left(W_{\k,\mu}; f\right)  = \sum_{k=0}^n \proj_k \left(W_{\k,\mu}; f\right).
$$
If $f(x,t)$ depends only on $t$, then $\wh f_{m,\kb}^n = 0$ for all $m > 0$, and the series is again reduced
to the classical Fourier-Laguerre series. 

\begin{prop}
Let $d \ge 1$, $\k \ge 0$ and $\mu \ge -\f12$. Define $\a = \mu + |\k|+\frac{d-1}{2}$. If $f(x,t) = f_0(t)$, 
where $f_0 \in L^2(\RR_+; w_{2\a})$. Then the Fourier-Laguerre series of $f$ on the cone is equal to 
the Fourier-Laguree series of $f_0$ in $L^2(\RR_+; w_{2\a})$. In particular, 
$$
     \Sb_n \left(W_{\k,\mu}; f, (x,t)\right) = s_n (w_{2 \a}; f_0, t), \qquad n= 0,1,2,\ldots. 
$$
\end{prop}

As in the case of conic surface, the polynomials $\Lb_{m,\kb}^{n}$ can also be deduced from taking the
limit of the Jacobi type orthogonal polynomials defined in \cite{X20a} on the solid cone. A more direct 
way, however, is to relate them to the orthogonal polynomials on the conic surface that we have already 
encountered. The relation is modeled after, in fact uses, the relation between orthogonal polynomials on
the unit ball and those on the unit sphere. 

For $(x,t) \in \VV^{d+1}$, we introduce the notation $X = (x,\sqrt{t^2-\|x\|^2})$, so that $\|X\| =t$ and 
$(X,t) \in \VV_0^{d+2}$. Similarly, for $(y,s) \in \VV^{d+1}$, we let $Y = (y,\sqrt{t^2-\|y\|^2}) \in \VV_0^{d+2}$. Define 
$\bkappa = (\k,\mu) \in \RR^{d+2}$. Then, for $x = t x'$, $x'\in \BB^d$, we can write
\begin{equation} \label{eq:weightV-V0}
 W_{\k,\mu} (x,t) \d x \d t = h_\bkappa^2(X) t^{-1} e^{-t} \frac{\d x' \d t}{\sqrt{1-\|x'\|^2}} = 
     \sw_{\kb}(X) \d \s (X,t),
\end{equation}
where $\d \s = \d \s_{\VV_0^{d+2}}$ is the Lebesgue measure on the surface $\VV_0^{d+2}$ in the right-hand side. 
Let 
$$
      \VV_{0,+}^{d+2} = \left\{(x, x_{d+1}, t) \in \VV_0^{d+2}: x_{d+1} \ge 0\right\},
$$ 
which is half of the conic surface $\VV_0^{d+2}$. Likewise, we define  $\VV_{0,-}^{d+2}$ as the other half
with $x_{d+1} < 0$. Let $X^- = (x,- \sqrt{t^2-\|x\|^2})$. Then $(x,t) \mapsto (X,t)$ maps $\VV^{d+1}$ onto 
$\VV_{0,+}^{d+2}$ and $(x,t) \mapsto (X^-,t)$ maps $\VV^{d+1}$ onto $\VV_{0,-}^{d+2}$.  By symmetry, 
it follows that 
\begin{align} \label{eq:OPVandV_0}
 \int_{\VV_{0}^{d+2}}   f(y,s)  \sw_{\bkappa}(y,s)\d \s (y,s) 
   =  \int_{\VV^{d+1}} \frac12 \left[f(X,t) + f(X^-,t) \right]W_{\k,\mu} (X,t) \d x \d t.     
\end{align}
As a consequence of this relation, we obtain a relation between orthogonal polynomials on the cone 
and the conic surface. To emphasis the dependence on the weight function, we denote the orthogonal
polynomial given in \eqref{eq:coneLsf} by $\Lb_{m,\kb}^n(W_{\k,\mu})$. 
Define 
$$
   Y_{m, \kb}^{n,1} (X,t) = \Lb_{m,\kb}^n(W_{\k,\mu}; (x,t)),  \qquad 
   Y_{m, \kb}^{n,2} (X,t) = X_{d+1} \Lb_{m,\kb}^{n-1}(W_{\k,\mu+1}; (x,t)) 
$$
for $X = (x, X_{d+1}) \in \RR^{d+1}$ and $(X,t) \in \VV_0^{d+2}$. 

\begin{prop} \label{prop:OPV-V0}
Let $\bkappa = (\k,\mu)$ and $X = (x, X_{d+1})$, $X_{d+1} = \sqrt{t^2 - \|x\|^2}$. Under the mapping 
$\VV^{d+1} \mapsto \VV_{0,+}^{d+2}: (x,t) \mapsto (X,t)$,  
$$
   \CV_n\left(\VV_0^{d+2}, \sw_{\bkappa}\right) = \CV_n\left(\VV_{0}^{d+1}, W_{\k,\mu} \right)\bigoplus
        X_{d+1} \CV_n \left(\VV_{0}^{d+1}, W_{\k,\mu+1} \right).
$$
More precisely, an orthogonal basis of  $\CV_n\left(\VV_0^{d+2}, \sw_{\bkappa}\right)$ is given by 
$$
\{Y_{\kb,m}^{n,1}: |\kb| =m, 0\le m \le n\}\cup \{ Y_{\kb,m}^{n,2}: |\kb| =m, 0\le m \le n-1\}.
$$
\end{prop}

\begin{proof}
Since $P_\kb^m$ is even when $m$ is even and odd when $m$ is odd, it follows that $t^m P_\kb^m(\frac{x}{t})$
is a homogeneous polynomial of degree $m$ in the variables $(x,t)$ and, in particular, a polynomials of 
degree $m$ in the variables $(X,t)$. Hence, it follows that $Y_{\kb,m}^{n,1}$ is a polynomial of 
degree $n$ in $(X,t)$ variables, even with respect to $X_{d+1}$ variable. Moreover, the same argument 
shows that $Y_{\kb,m}^{n,2}$ is a polynomial of degree $n$ in $(X,t)$ variables, odd in $X_{d+1}$ variable. 
By \eqref{eq:OPVandV_0}, it follows that $Y_{\kb,m}^{n,1}$ and $Y_{\kb,m}^{n,2}$ are orthogonal with 
respect to $\la \cdot,\cdot \ra_{\bkappa}$ on $\VV_0^{d+2}$ by parity and, moreover, by \eqref{eq:weightV-V0} 
and $W_{\k,\mu+1} (x,t) \d x \d t = X_{d+1}^2 \sw_{\kb}(X) \d \s (X,t)$ followed by \eqref{eq:weightV-V0},
\begin{align*}
     \big \langle  Y_{\kb,m}^{n,2}, Y_{\kb',m'}^{n',2} \big \rangle_{\bkappa} \,& = 
    \big \langle \Lb_{m,\kb}^n(W_{\k,\mu}), \Lb_{m',\kb'}^{n'}(W_{\k,\mu}) \big \rangle_{\k,\mu} \\
   \big \langle  Y_{\kb,m}^{n,1}, Y_{\kb',m'}^{n',1} \big \rangle_{\bkappa} 
     \,& =  \big \langle   \Lb_{m,\kb}^{n-1}(W_{\k,\mu+1}),  \Lb_{m',\kb'}^{n'-1}(W_{\k,\mu+1}) \big \rangle_{\k,\mu+1},
\end{align*}
which shows that $ \Lb_{m,\kb}^n(W_{\k,\mu})$ and $X_{d+1} \Lb_{m,\kb}^{n-1}(W_{\k,\mu+1})$ are elements
of $\CV_n(\VV_0^{d+2}, \sw_\bkappa)$. Furthermore, it is easy to verify that 
$$
  \dim \CV_n\left(\VV_0^{d+2}, \sw_\bkappa\right) = \dim \CV_n\left(\VV^{d+1}, W_{\k,\mu}\right)+ 
  \dim \CV_{n-1} \left(\VV^{d+1}, W_{\k,\mu+1}\right).
$$
This completes the proof. 
\end{proof}

\subsection{Reproducing and Poisson kernels}
Let $\Pb_n(W_{\k,\mu}; \cdot, \cdot)$ be the reproducing kerne of $\CV_n(\VV^{d+1}, W_{\k,\mu})$. It is uniquely
determined by 
$$
  \bb_{\k,\mu} \int_{\VV^{d+1}} P(y,s) \Pb_n\left(W_{\k,\mu}; \cdot, (y,s)\right)W_{\k,\mu}(y, s) \d y \d s = P, 
\quad \forall P \in \CV_n(\VV^{d+1}(W_{\k,\mu}).
$$

\begin{thm} \label{thm:repordVL}
Let $d \ge 1$ and $\mu \ge 0$. Let $\a =\a_{\k,\mu} = |\k|+ \mu+\f{d-1}{2} > 0$. Then for $(x,t), (y,s) \in \VV^{d+1}$, 
\begin{align} \label{eq:PnLV}
 & \Pb_n  \big(W_{\k,\mu}; (x,t),(y,s)\big) = C_{\k,\mu}  \int_{[1,1]^{d+1}} \int_0^\pi 
     L_n^{2 \a} \big(t+s+2 \brho(x,t,y,s;u) \cos \t \big)  \\
      \times & e^{-\brho (x,t,y,s;u) \cos \t} j_{\a-1} \big(\brho (x,t,y,s;u) \sin \t\big) (\sin \t)^{2\a -1}\d\t 
        \Phi_\k(u') (1-u_{d+1}^2)^{\mu -1}\d u,  \notag
\end{align}
where $C_{\k} =  \frac{2^{\a -1} \Gamma(\a +\f12)}{\sqrt{\pi}} c_{\k-\f{\mathbf{1}}{2}}
c_{\mu-\f{1}{2}}$, $u=(u',u_{d+1})$ and 
$$
    \brho (x,t,y,s;u)  = \sqrt{ \tfrac12 \left(t s + x_1y_1 u_1+\cdots +x_dy_d u_d 
         + \sqrt{t^2-\|x\|^2} \sqrt{s^2-\|y\|^2} u_{d+1}\right)},
$$ 
and \eqref{eq:PnLV} holds under the limit \eqref{eq:limitInt} if either $\a =0$ or $\k_i =0$ for
one or more $i$.  
\end{thm} 

\begin{proof}
By Proposition \ref{prop:OPV-V0}, the space $\CV_n(\VV^{d+1}, W_{\k,\mu})$ corresponds to the subspace 
of  $\CV_n(\VV_0^{d+2}, \sw_{\bkappa})$ that consists of polynomials even in $X_{d+1}$ variable. By symmetry,
it follows readily that 
$$
   \Pb_n\left(W_{\k,\mu}; (x,t), (y,s)\right) = \frac12 \left[ \sP_n\left(\sw_{\bkappa}; (X,t), (Y,s)\right)
        + \sP_n\left(\sw_{\bkappa}; (X,t), (Y^-,s)\right) \right], 
$$
where $\bkappa = (\k_1,\ldots, \k_{d+1})$ and $\sP_n(\sw_{\bkappa}; \cdot,\cdot)$ is the reproducing 
kernel of $\CV_n(\VV^{d+2}, \sw_{\bkappa})$. Consequently, \eqref{eq:PnLV} follows from 
\eqref{eq:PnLV0} with $d+1$ replaced by $d+2$, where the integral with respect to $\d u_{d+1}$
losses $(1+u_{d+1})$ factor because of the symmetry. 
\end{proof}

We can also state such a closed formula when $\a_\k =0$ by using \eqref{eq:PnLV0_d=2}. Since $\a_\k = 0$ 
is equivalent to $\k = 0$, $\mu =0$ and $d =1$, which is the degenerate case that is equivalent to the product 
Laguerre weight $w_{-\f12}(x_1)w_{-\f12}(x_2)$ on $\RR_+^2$ by Remark \ref{rem:5.1}, we shall not write 
down the formula. 

We can also define the Poisson kernel of the orthogonal expansion on $\VV^{d+1}$ by
$$
 \Pb \big(W_{\k,\mu}; r, (x,t),(y,s)\big) = \sum_{n=0}^\infty  \Pb_n  \big(W_{\k,\mu}; (x,t),(y,s)\big) r^n. 
$$
Then, as an analogue of Theorem \ref{thm:PoissonV0L}, we obtain a closed form formula.

\begin{thm} \label{thm:PoissonVL}
Under the same assumption as in Theorem \ref{thm:repordVL}, 
\begin{align*}
 \Pb(\sw_{\k,\mu}; r, (x,t), (y,s))  =
   \, & \frac{\e^{ -\frac{(t+s)r}{1-r}}}{(1-r)^{2\a_{\k,\mu} +1}}  
     \int_{[-1,1]^{d+2}}   \exp\bigg\{\frac{2 \sqrt{r}  v }{1-r} \brho(x,t,y,s;u)\bigg\} \\
        & \times c_{\a_{\k,\mu}-\f12} c_{\bkappa-\f{\mathbf{1}}{2}} \Phi_\k(u') (1-u_{d+1}^2)^{\mu-1} \d u
           (1-v^2)^{\a_{\k,\mu}-1} \d v,
\end{align*}
and it holds under the limit \eqref{eq:limitInt} if either $\a_\k =0$ or $\k_i =0$ for one or more $i$.
\end{thm}

\subsection{Pseudo convolution on the cone}
We can also define a pseudo convolution on the solid cone $\VV^{d+1}$, which will be bounded in the space 
$L^p_{\ub(\k,\mu)}(\VV^{d+1})$, defined as the space of functions with finite $\|f\|_{p, \k,\mu}$ norm, where
$$
 \|f\|_{p, \k,\mu}:= \left(\bb_{\k,\mu} \int_{\VV^{d+1}}| f(x,t) \e^{-t/2}|^p \ub(z,t) \d x \d t 
   \right)^{\f1p},
$$
for $1 \le p < \infty$ and $\ub(x,t) = h_\k^2(x) (t^2-\|x\|^2)^{\mu-\f12}$, and 
$$
   \|f\|_{\infty} =  \|f\|_{\infty, \k,\mu}: = \esssup\left \{f(x,t) \e^{-t/2}: (x,t) \in \VV^{d+1} \right \}. 
$$

Using the notation of Theorem \ref{thm:repordVL} with $\a = \a_{\k,\mu} > 0$, we define, for 
$\in L^1(\RR_+, \varpi_{2\a_{\k,\mu}})$ and $(x,y) \in \VV^{d+1}$, the generalized translation operator 
\begin{align*}
 \Tb_{(x,t)} g(y,s) = \, & C_{\k,\mu} \int_{[-1,1]^{d+1}}\int_{0}^\pi g(t+s+2 \brho(x,t,y,s;u) \cos \t \big)
   e^{- \brho (x,t,y,s;u)\cos \t}  \\
       &\times  j_{\a_\k -1} \big(\brho (x,t,y,s;u) \sin \t\big) (\sin \t)^{2\a_{\k,\mu} -1}\d\t
        \Phi_\k(u') (1-u_{d+1}^2)^{\mu-1}  \d u. 
        \notag
\end{align*}
By its definition and \eqref{eq:PnLV}, the reproducing kernel $\Pb_n(W_{k,\mu})$ can be written as 
\begin{equation} \label{eq:TLn=PnV}
    \Pb_n(W_{\k,\mu}; (x,t),\cdot ) =  \Tb_{(x,t)} L_n^{2|\k|+2\mu+d-1}, \qquad n = 0, 1, 2, \ldots, 
\end{equation}

\begin{prop} \label{lem:Tbd}
Let $d \ge 1$ and $\a_{\k,\mu} = \mu+ |\k|+\frac{d-1}{2} \ge \f12$. For $g\in L^p_{u(2\a_{\k,\mu})}(\RR_+)$,
$$
   \| \Tb_{(x,t)} g\|_{p, \k,\mu} \le  e^{t/2} \|g\|_{p,  u(2\a_\k)}, \qquad 1 \le p \le \infty.
$$
\end{prop}

\begin{proof}
Let $\sT_{(X,t)}$ denote the generalized translation on the conic surface $\VV_0^{d+2}$, defined as
in Definition \ref{def:sT-transl} but with $\sw_\k$ replaced by $\sw_{\bkappa}$ and $\rho(x,t,y,s;u)$ 
replaced by $\brho(x,t,y,s; u)$. Comparing the definitions, it follows that 
\begin{equation} \label{eq:Tb=sT+sT}
    \Tb_{(x,t)} g(y,s) = \frac12 \left[\sT_{(X,t)}(Y,s) + \sT_{(X,t)}(Y^-,s)\right] 
\end{equation}
and we also have $\sT_{(X,t)}(Y^-,s) = \sT_{(X^-,t)}(Y,s))$. Hence, by \eqref{eq:OPVandV_0}, the
boundedness of $\Tb_{(x,t)}$ as stated follows from the boundedness of $\sT_{(x,t)}$ in 
Proposition \ref{prop:sTbd}.
\end{proof}

For $\a_{\k,\mu} < \f12$, we have a counterpart of Proposition \ref{prop:sTbd2}. 

\begin{prop} \label{lem:Tbd2}
Let $d \ge 1$ and $\a_{\k,\mu} \le \f12$. For $g\in L^1_{u(2\a_{\k,\mu})}(\RR_+)\cap L^1_{u(2\a_{\k,\mu}+\f12)}(\RR_+)$,
$$
   \| \Tb_{(x,t)} g\|_{1, \k,\mu} \le  e^{t/2} \|g\|_{1,  u(2\a_{\k,\mu})}+  \sqrt{t} e^{t/2} \|g\|_{1,  u(2\a_{\k,\mu}+\f12)}
      , \qquad 1 \le p \le \infty.
$$
\end{prop}

The generalized operator is used to defined a pseudo convolution operator on $\VV^{d+1}$.  
For $f \in L^2_{\ub(\k,\mu)}(\VV^{d+1})$ and $g \in L^2_{2 u(\a_{\k,\mu})}(\RR_+)$, we define
$$
 f *_\VV g(x,t) = \bb_{\k,\mu} \int_{\VV^{d+1}} f(y,s) \Tb_{(x,t)}g (y,s) W_{\k,\mu}(y,s) \d y \d s, \quad (x,t) \in \VV^{d+1}.
$$

\begin{thm} \label{thm:f*g-bddV}
Let $d \ge 1$ and $\a_{\k,\mu} \ge \f12$. For $f \in L^p_{\ub(\k,\mu)}(\VV^{d+1})$, $1 \le p\le \infty$, and 
$g\in L^1_{u(2\a_{\k,\mu})}(\RR_+)$,
$$
  \left \|f*_\VV g \right \|_{p, \k,\mu} \le \|f\|_{p,\k,\mu} \|g\|_{1, u(2\a_{\k,\mu})}, \qquad 1 \le p \le \infty. 
$$
\end{thm}

\begin{proof}
Let $f*_{\VV_0}g$ denote the pseudo convolution on $\VV_0^{d+2}$ defined as in Definition \ref{def:f*g}
with $\k$ replaced by $\bkappa$ and $d+1$ replaced by $d+2$. By \eqref{eq:Tb=sT+sT}, it follows
that 
$$
  f *_\VV g(x,t) =  \f12\left[ f*_{\VV_0} g(X,t)+ f*_{\VV_0} g(X^-,t) \right]
$$
Hence, by \eqref{eq:OPVandV_0}, the boundedness of $f *_\VV g$ follows from the boundedness 
of $f*_{\VV_0} g$ in Theorem \ref{thm:f*g-bdd}.
\end{proof}

Let the space  $L^p_{\ub(\k,\mu),*}(\VV^{d+1})$ be defined with an additional $\sqrt{t}$ in its norm 
$\|f\|_{p, \k,\mu}^*$, in analogous to the space  $L^p_{\su(\k),*}(\VV_0^{d+1})$. We also have an 
analogue of Theorem \ref{thm:f*g-bdd2}. 

\begin{thm} \label{thm:f*g-bddV2}
Let $d \ge 2$ and $\a_{\k,\mu} < \f12$. For $f \in L^1_{\ub(\k,\mu)}(\VV^{d+1})\cap L^1_{\ub(\k,\mu),*}(\VV^{d+1})$
and $g\in L^1_{u(2\a_{\k,\mu})}(\RR_+)\cap L^1_{u(2\a_{\k,\mu}+\f12)}(\RR_+)$,   
$$
  \left \|f*_\VV g \right \|_{1, \k,\mu} \le \|f\|_{1,\k,\mu} \|g\|_{1, u(2\a_{\k,\mu})} + \|f\|_{1,\k,\mu}^* 
        \|g\|_{1, u(2\a_{\k,\mu}+\f12)}.
$$
\end{thm}

\subsection{Ces\`aro means of Fourier-Laguerre series on the conic surface}
For $\delta > -1$, the Ces\`aro means of the Fourier-Laguerre series is given by
\begin{align*}
  \Sb_n^\delta(W_{\k,\mu}; f) \, & = \frac{1}{\binom{n+\delta}{n}} \sum_{k=0}^n 
      \binom{n-k+\delta}{n-k} \proj_k (W_{\k,\mu}; f).
\end{align*}

\begin{thm}\label{thm:CesaroV}
Let $d \ge 1$ and $\mu \ge 0$. Assume $\a_{\k,\mu} = |\k|+\mu+ \f{d-1}{2} \ge \f12$. Then 
$\Sb_n^\delta(W_{\k,\mu}; f)$ converse to $f$ in $L^p_{\ub(\k,\mu)}(\VV^{d+1})$, $1\le p \le \infty$, if 
$\delta > 2|\k|+2\mu+d- \f12$ and the inequality is sharp if $p =1$ and $p = \infty$. 
\end{thm}

\begin{proof}
As in the proof of Theorem \ref{thm:CesaroV0}, we need the boundedness of 
$$
  \left \|\Sb_n^\delta(W_{\k,\mu}) \right \|_{p,\k,\mu} =
     \int_{\VV^d} \left| \Kb_n^\delta \left(W_{\k,\mu}; (x,t),(y,s) \right)\right | \e^{s/2} W_{\k,\mu}(y,s) \d y \d s
$$
where $p=1$ or $p = \infty$ and $ \Kb_n^\delta \left(W_{\k,\mu} \right)$ is the kernel of the $(C,\delta)$ 
means, which can be written in terms of its counterpart $\sK_n^\delta(\sw_{\bkappa})$, with $\bkappa = (\k,\mu)$,  
on $\VV_0^{d+2}$ as
$$
 \Kb_n^\delta \left(W_{\k,\mu}; (x,t),(y,s) \right) = \frac 12 \left[\sK_n^\delta\left(\sw_{\bkappa}; (X,t), (Y,s)\right)+ 
\sK_n^\delta\left(\sw_{\bkappa}; (X,t), (Y^-,s)\right) \right],
$$
so that the boundedness of $\left \|\Sb_n^\delta(W_{\k,\mu}) \right \|_{p,\k,\mu}$ from above follows from
Theorem \ref{thm:CesaroV0}. Furthermore, since $ \Kb_n^\delta \left(W_{\k,\mu}; (x,t),(0,0) \right) 
= \sK_n^\delta\left(\sw_{\bkappa}; (X,t), (0,0)\right)$, the boundedness from below also follows by 
the proof of Theorem \ref{thm:CesaroV0}. 
\end{proof}

For $\a_{\k,\mu} \le \f12$, we can state the following counterpart of Theorem \ref{thm:Cesaro2}.

\begin{thm} \label{thm:Cesaro2V}
Let $d \ge 1$ and $\mu \ge 0$.  Assume $\a_{\k,\mu} \le \f12$. Then $\Sb_n^\delta(W_{\k,\mu}; f)$ converse 
to $f$ in $L^1_{\ub(\k,\mu)}(\VV^{d+1})\cap L^1_{\ub(\k,\mu),*}(\VV^{d+1})$ if $\delta > 2 \a_{\k,\mu}+1$. 
Moreover, $\Sb_n^\delta(W_{\k,\mu}; f)$ does not converge for all $f \in L^1_{\ub(\k,\mu)}(\VV^{d+1})$ if 
$\delta \le 2\a_{\k,\mu}+\f12$. 
\end{thm}

As discussed in Remark \ref{rem:5.1}, when $d =1$, our set-up is equivalent to the orthogonal expansions
with respect to $U_{\k,\mu}$, defined in \eqref{eq:U-weight}, on $\RR_+^2$, which becomes product Laguerre
expansions if $\k=0$; the case $\k > 0$ has not been studied as far as we are aware. 
Let $L^p_{u(\k,\mu)}(\RR_+^2)$ be the space with the norm
$$
   \|f\|_{p,u(\k,\mu)} = \left ( \int_{\RR_+^2} |f(x_1,x_2)|\e^{(x_1+x_2)/2} U_{\k,\mu}(x_1,x_2) \d x_1 \d x_2 \right)^{\f1p},
 \quad 1 \le p <\infty,
$$  
and 
$$
  \|f\|{\infty,u(\k,\mu)} = \esssup \left\{ | f(x_1,x_2)| \e^{- (x_1+x_2)/2}: (x_1,x_2) \in \RR_+^2 \right \}.
$$
Let $S_n^\delta(U_{\k,\mu}; f)$ be the Ces\`aro means of the Fourier-orthogonal expansions with respect 
to $U_{\k,\mu}$ on $\RR_+^2$. Then our Theorem \ref{thm:CesaroV} for $d=2$ yields: 

\begin{cor}
Let $\k, \mu \ge 0$. If $\k+ \mu \ge \f12$, then $S_n^\delta(U_{\k,\mu}; f)$ converges to $f \in 
L^p_{u(\k,\mu)}(\RR_+^2)$, $1\le p \le \infty$, if $\delta > 2\k+2\mu+ \f12$ and the result is sharp
if $p =1$ or $p = \infty$. 
\end{cor}

If $\k + \mu < \f12$, we can also state a corollary of Theorem \ref{thm:Cesaro2V} that gives a sufficient 
condition $\delta > 2 \k+ 2\mu +1$ for the convergence in $L^1$ norm. For $\k = 0$, the result for 
$\mu \ge \f12$ agrees with the product Lagueree expansions for $w_\a(x_1)w_\b(x_2)$ with $\a = \b =
\mu - \f12 \ge 0$ (cf. \cite[Theorem 2.3]{X00}); while the result for $\mu < \f12$ provides a sufficient 
condition for the product Laguerre expansions with $\a=\b < 0$, which however is likely not sharp.


\begin{thebibliography}{99}
	
\bibitem{AW}
      R. Askey and S. Wainger, 
      Mean convergence of expansions in Laguerre and Hermite series, 
      \textit{Amer. J. Math. } \textbf{87} (1965), 695--708. 
      
\bibitem{BRT}
       P. Boggarapu, L. Roncal and S. Thangavelu,  
       Mixed norm estimates for the Ces\`aro means associated with Dunkl-Hermite expansions.
       \textit{Trans. Amer. Math. Soc.} \textbf{369} (2017), 7021--7047.

\bibitem{DLMF} 
       {\it NIST Digital Library of Mathematical Functions}. http://dlmf.nist.gov/
        	
\bibitem{DaiX} 
       F. Dai and Y. Xu, 
       \textit{Approximation theory and harmonic analysis on spheres and balls}.
       Springer Monographs in Mathematics, Springer, 2013. 
        
\bibitem{D89} 
       C. F. Dunkl,
       Differential-difference operators associated to reflection groups.
       \textit{Trans. Amer. Math. Soc.} \textbf{311} (1989), 167--183.
        
\bibitem{DX} 
       C. F. Dunkl and Y. Xu,
       \textit{Orthogonal Polynomials of Several Variables}
       Encyclopedia of Mathematics and its Applications \textbf{155},
       Cambridge University Press, Cambridge, 2014.
   
\bibitem{GM}
        E. G\"orlich and C. Markett,
        A convolution structure for Laguerre series, 
        \textit{Indag. Math.}, \textbf{44} (1982), p. 161--171.

\bibitem{Ma}
        C. Markett, 
        Mean Ces\`aro summability of Laguerre expansions and norm estimates with shifted parameter, 
        \textit{Analysis Math.}, \textbf{8} (1982), 19--37. 

\bibitem{Me}
        C. Meaney, 
        Divergent Ces\`aro and Riesz means of Jacobi and Laguerre expansions. 
        \textit{Proc. Amer. Math. Soc.} \textbf{131} (2003), 3123--3128.

\bibitem{MW}
         B. Muckenhoupt and D. Webb,
         Two-Weight Norm Inequalities for Ces\`aro Means of Laguerre Expansions.
         \textit{Trans. Amer. Math. Soc.} \textbf{353} (2001), 1119--1149.

\bibitem{NS}
          A. Nowak and K. Stempak, 
          Riesz transforms for multi-dimensional Laguerre function expansions. 
          \textit{Adv. Math.} \textbf{215} (2007), 642--678. 

\bibitem{NSS}
          A. Nowak, K. Stempak and T. Szarek, 
          On harmonic analysis operators in Laguerre-Dunkl and Laguerre-symmetrized settings. 
          {\it SIGMA} \textbf{12} (2016), Paper No. 096, 39 pp. 

\bibitem{OX} 
          S. Olver and Y. Xu,
          Non-homogeneous wave equation on a cone. 
          \textit{Integral Transforms Spec Funct.}, to appear.

\bibitem{Po} 
         Eileen Poiani, 
         Mean Ces\`aro summability of Laguerre and Hermite series, 
         \textit{Trans. Amer. Math. Soc.} \textbf{173} (1972), 1--31.

\bibitem{RT} 
         R. Radha and S. Thangavelu, 
         Hardy's inequalities for Hermite and Laguerre expansions. 
         \textit{Proc. Am. Math. Soc.} \textbf{132}  (2004), 3525--3536.
 
\bibitem{Sz} 
         G. Szeg\H{o},
         \textit{Orthogonal polynomials}. 4th edition,
         Amer. Math. Soc., Providence, RI. 1975
   
\bibitem{Th0}  
         S. Thangavelu, 
         Summability of Laguerre expansions, 
         \textit{Analysis Math.} \textbf{16} (1990), 303--315.
\bibitem{Th} 
         S. Thangavelu, 
         \textit{Lectures on Hermite and Laguerre Expansions},
         Princeton Univ. Press,  Princeton, NJ, 1993. 

\bibitem{Wa}
         G. N. Watson,
         Another note on Laguerre polynomials, 
         \textit{J. London Math. Soc.}, \textbf{14} (1939), 19--22.

\bibitem{X97b} 
         Y. Xu, 
         Integration of the intertwining operator for h-harmonic polynomials associated to reflection groups.
         \textit{Proc. Amer. Math. Soc.}, \textbf{125} (1997), 2963--2973.

\bibitem{X00}
          Y. Xu, 
          A note on summability of Laguerre expansions, 
          \textit{Proc. Amer. Math. Soc.}, \textbf{128} (2000), 3571--3578.
          
\bibitem{X20a}          
          Y. Xu,
          Orthogonal polynomials and Fourier orthogonal series on a cone. 
          \textit{J. Fourier Anal. Appl.} \textbf{26} (2020), Article number:36 

\bibitem{X20b}          
          Y. Xu,    
          Orthogonal structure and orthogonal series in and on a double cone or a hyperboloid. 
          \textit{Trans. Amer. Math. Soc.} in print. 

\end{thebibliography}
\end{document}